\documentclass[fleqn,reqno,11pt,a4paper,final]{amsart}
\usepackage[a4paper,left=35mm,right=35mm,top=35mm,bottom=35mm,marginpar=25mm]{geometry} 
\usepackage{amsmath}
\usepackage{amssymb}
\usepackage{amsthm}
\usepackage{amscd}
\usepackage[ansinew]{inputenc}
\usepackage{cite}
\usepackage{bbm}
\usepackage{xcolor}
\usepackage[english=american]{csquotes}
\usepackage[final]{graphicx}
\usepackage{hyperref}
\usepackage{calc}
\usepackage{bm}
\usepackage{enumerate}

\graphicspath{{Pictures/}}

\numberwithin{equation}{section}

\newtheoremstyle{thmlemcorr}{10pt}{10pt}{\itshape}{}{\bfseries}{.}{10pt}{{\thmname{#1}\thmnumber{ #2}\thmnote{ (#3)}}}
\newtheoremstyle{thmlemcorr*}{10pt}{10pt}{\itshape}{}{\bfseries}{.}\newline{{\thmname{#1}\thmnumber{ #2}\thmnote{ (#3)}}}
\newtheoremstyle{defi}{10pt}{10pt}{\itshape}{}{\bfseries}{.}{10pt}{{\thmname{#1}\thmnumber{ #2}\thmnote{ (#3)}}}
\newtheoremstyle{remexample}{10pt}{10pt}{}{}{\bfseries}{.}{10pt}{{\thmname{#1}\thmnumber{ #2}\thmnote{ (#3)}}}
\newtheoremstyle{ass}{10pt}{10pt}{}{}{\bfseries}{.}{10pt}{{\thmname{#1}\thmnumber{ A#2}\thmnote{ (#3)}}}

\theoremstyle{thmlemcorr}
\newtheorem{theorem}{Theorem}
\numberwithin{theorem}{section}
\newtheorem{lemma}[theorem]{Lemma}
\newtheorem{corollary}[theorem]{Corollary}
\newtheorem{proposition}[theorem]{Proposition}

\theoremstyle{thmlemcorr*}
\newtheorem{theorem*}{Theorem}
\newtheorem{lemma*}[theorem]{Lemma}
\newtheorem{corollary*}[theorem]{Corollary}
\newtheorem{proposition*}[theorem]{Proposition}
\newtheorem{problem*}[theorem]{Problem}
\newtheorem{conjecture*}[theorem]{Conjecture}

\theoremstyle{defi}

\theoremstyle{remexample}
\newtheorem{remark}[theorem]{Remark}

\theoremstyle{ass}

\newcommand{\Crm}{\mathrm{C}}

\newcommand{\Lrm}{\mathrm{L}}

\newcommand{\Wrm}{\mathrm{W}}

\newcommand{\Ecal}{\mathcal{E}}

\newcommand{\Hcal}{\mathcal{H}}

\newcommand{\Lcal}{\mathcal{L}}
\newcommand{\Mcal}{\mathcal{M}}

\newcommand{\Ebf}{\mathbf{E}}

\newcommand{\Ybf}{\mathbf{Y}}

\renewcommand{\Bbb}{\mathbb{B}}

\newcommand{\Sbb}{\mathbb{S}}

\DeclareMathOperator{\id}{id}

\DeclareMathOperator{\supmod}{sup}

\DeclareMathOperator{\supp}{supp}
\newcommand{\ee}{\mathrm{e}}

\newcommand{\setn}[2]{\{\, #1 \ \ \textup{\textbf{:}}\ \ #2 \,\}}
\newcommand{\setb}[2]{\bigl\{\, #1 \ \ \textup{\textbf{:}}\ \ #2 \,\bigr\}}

\newcommand{\setBB}[2]{\biggl\{\, #1 \ \ \textup{\textbf{:}}\ \ #2 \,\biggr\}}

\newcommand{\norm}[1]{\|#1\|}

\newcommand{\abs}[1]{|#1|}

\newcommand{\absn}[1]{|#1|}
\newcommand{\absb}[1]{\bigl|#1\bigr|}

\newcommand{\absBB}[1]{\biggl|#1\biggr|}

\newcommand{\dpr}[1]{\langle #1 \rangle}	

\newcommand{\dprn}[1]{\langle #1 \rangle}
\newcommand{\dprb}[1]{\bigl\langle #1 \bigr\rangle}

\newcommand{\ddpr}[1]{\langle\!\langle #1 \rangle\!\rangle}

\newcommand{\ddprn}[1]{\langle\!\langle #1 \rangle\!\rangle}
\newcommand{\ddprb}[1]{\bigl\langle\hspace{-2.5pt}\bigl\langle #1 \bigr\rangle\hspace{-2.5pt}\bigr\rangle}

\newcommand{\ddprBB}[1]{\biggl\langle\!\!\!\biggl\langle #1 \biggr\rangle\!\!\!\biggr\rangle}

\newcommand{\cl}[1]{\overline{#1}}
\newcommand{\di}{\mathrm{d}}
\newcommand{\dd}{\;\mathrm{d}}

\newcommand{\N}{\mathbb{N}}
\newcommand{\R}{\mathbb{R}}

\newcommand{\sym}{\mathrm{sym}}
\newcommand{\skw}{\mathrm{skew}}

\newcommand{\reg}{\mathrm{reg}}
\newcommand{\sing}{\mathrm{sing}}
\newcommand{\ONE}{\mathbbm{1}}

\newcommand{\toweakstar}{\overset{*}\rightharpoonup}

\newcommand{\todown}{\downarrow}

\newcommand{\conv}{\star}

\newcommand{\sbullet}{\begin{picture}(1,1)(-0.5,-2)\circle*{2}\end{picture}}
\newcommand{\frarg}{\,\sbullet\,}
\newcommand{\BV}{\mathrm{BV}}
\newcommand{\BD}{\mathrm{BD}}

\newcommand{\BDY}{\mathbf{BDY}}

\newcommand{\toY}{\overset{\Ybf}{\to}}

\newcommand{\eps}{\epsilon}

\DeclareMathOperator{\Tan}{Tan}

\newcommand{\term}[1]{\textbf{#1}}

\newcommand{\proofstep}[1]{\textit{#1}}

\newcommand{\Rdds}{\R^{d \times d}_\sym}
\newcommand{\Bdds}{\Bbb^{d \times d}_\sym}

\def\Xint#1{\mathchoice 
{\XXint\displaystyle\textstyle{#1}}%
{\XXint\textstyle\scriptstyle{#1}}%
{\XXint\scriptstyle\scriptscriptstyle{#1}}%
{\XXint\scriptscriptstyle\scriptscriptstyle{#1}}%
\!\int} 
\def\XXint#1#2#3{{\setbox0=\hbox{$#1{#2#3}{\int}$} 
\vcenter{\hbox{$#2#3$}}\kern-.5\wd0}} 
\def\dashint{\,\Xint-}

\newcommand{\restrict}{\begin{picture}(10,8)\put(2,0){\line(0,1){7}}\put(1.8,0){\line(1,0){7}}\end{picture}}

\renewcommand{\eps}{\varepsilon}
\renewcommand{\phi}{\varphi}

\title[Characterization of BD-Young measures]{Characterization of generalized Young measures\\ generated by symmetric gradients}

\author[G.~De Philippis]{Guido De Philippis}
\address{\textit{G.~De Philippis:} SISSA, Via Bonomea 265, 34136 Trieste, Italy.}
\email{guido.dephilippis@sissa.it}

\author[F.~Rindler]{Filip Rindler}
\address{\textit{F.~Rindler:} Mathematics Institute, University of Warwick, Coventry CV4 7AL, UK.}
\email{F.Rindler@warwick.ac.uk}

\hypersetup{
  pdfauthor = {Guido De Philippis (SISSA) and Filip Rindler (University of Warwick)},
  pdftitle = {Characterization of generalized Young measures generated by symmetric gradients},
  pdfsubject = {},
  pdfkeywords = {}
}

\begin{document}

\maketitle


\begin{abstract}
This work establishes a characterization theorem for (generalized) Young measures generated by symmetric derivatives of functions of bounded deformation (BD) in the spirit of the classical Kinderlehrer--Pedregal theorem. Our result places such Young measures in duality with symmetric-quasiconvex functions with linear growth. The ``local'' proof strategy combines blow-up arguments with the singular structure theorem in BD (the analogue of Alberti's rank-one theorem in BV), which was recently proved by the authors. As an application of our characterization theorem we show how an atomic part in a BD-Young measure can be split off in generating sequences.
\vspace{4pt}

\noindent\textsc{MSC (2010): 49J45 (primary); 28B05, 46G10} 

\noindent\textsc{Keywords:} Young measure characterization, BD, symmetric-quasiconvexity.

\vspace{4pt}

\noindent\textsc{Date:} \today{}. 
\end{abstract}


\section{Introduction}

Young measures quantitatively describe the asymptotic oscillations in $\Lrm^p$-weakly converging sequences. They were introduced in~\cite{Young37,Young42a,Young42b} and later developed into an important tool in modern PDE theory and the calculus of variations in~\cite{Tartar79,Tartar83,Ball89,BallJames87} and many other works. In order to deal with concentration effects as well, DiPerna \& Majda extended the framework to so-called \enquote{generalized} Young measures, see~\cite{DiPernaMajda87,AlibertBouchitte97,KruzikRoubicek97,FonsecaMullerPedregal98,Sychev99,KristensenRindler10YM}. In the following we will refer also to these objects simply as \enquote{Young measures}.

When considering generating sequences that satisfy a differential constraint like curl-freeness (i.e.\ the generating sequence is a sequence of \emph{gradients}), the question immediately arises to characterize the resulting class of Young measures. In applications, these results provide very valuable information on the allowed oscillations and concentrations that are possible under this differential constraint, which usually constitutes a strong restriction.

The first general classification results are due to Kinderlehrer \& Pedregal~\cite{KinderlehrerPedregal91,KinderlehrerPedregal94}, who characterized classical \emph{gradient} Young measures, i.e.\ those generated by gradients of $\Wrm^{1,p}$-bounded sequences, $1 < p \leq \infty$. Their theorems put such gradient Young measures in duality with quasiconvex functions as introduced by Morrey~\cite{Morrey52}. For generalized Young measures the corresponding result was proved in~\cite{FonsecaMullerPedregal98} (also see~\cite{KalamajskaKruzik08}) and numerous other characterization results in the spirit of the Kinderlehrer--Pedregal theorems have since appeared, see for instance~\cite{KruzikRoubicek96,FonsecaMuller99,FonsecaKruzik10,BenesovaKruzik16}.

Characterization theorems are of particular use in the relaxation of minimization problems for non-convex integral functionals, where one passes from a functional defined on functions to one defined on Young measures. A Kinderlehrer--Pedregal-type theorem allows to restrict the class of Young measures over which to minimize. This is explained in detail (for classical Young measures) in~\cite{Pedregal97book}. A similar application is possible for generalized Young measures.

The characterization of generalized BV-Young measures was first achieved in~\cite{KristensenRindler10YM}. A different, \enquote{local} proof was given in~\cite{Rindler14YM}, another improvement is in~\cite[Theorem~6.2]{KirchheimKristensen16}. All of these arguments crucially use Alberti's rank-one theorem~\cite{Alberti93} (see~\cite{MassaccesiVittone16} for a short and elegant new proof) and thus, since this theorem is specific to BV, extensions to further BV-like spaces have been prohibited so far. The only partial result for a characterization beyond BV seems to be in~\cite{BaiaMatiasSantos13}, but that result is limited to first-order operators (which does not cover BD) and also additional technical conditions have to be assumed.

We now explain briefly the framework underlying this work and introduce some notation to state our main result; precise definitions are given in Section~\ref{sc:setup}. Given an $\Lrm^1$-bounded sequence of maps $v_j \colon \Omega \to \R^N$ ($\Omega \subset \R^d$), the Fundamental Theorem of (generalized) Young measure theory states that there exists a subsequence of the $v_j$'s (which we do not relabel) such that for all continuous $f \colon \cl{\Omega} \times \R^N \to \R$ with the property that the \term{recession function}
\[
  f^\infty(x,A) := \lim_{\substack{\!\!\!\! x' \to x \\ \!\!\!\! A' \to A \\ \; t \to \infty}}
    \frac{f(x',tA')}{t},  \qquad x \in \cl{\Omega}, \, A \in \R^N,
\]
exists, it holds that
\[
  \int f(x,v_j(x)) \dd x \to \ddprb{f,\nu}
  := \int_\Omega \dprb{f(x,\frarg), \nu_x} \dd x
    + \int_{\cl{\Omega}} \dprb{f^\infty(x,\frarg),\nu_x^\infty} \dd \lambda_\nu(x),
\]
where $(\nu_x)_{x \in \Omega}$, $(\nu_x^\infty)_{x \in \cl{\Omega}}$ are parametrized families of probability measures on $\R^N$ and $\partial \Bbb^N$ (the unit sphere in $\R^N$), respectively, and $\lambda_\nu$ is a positive, finite Borel measure on $\cl{\Omega}$. Together, we call $\nu = (\nu_x,\lambda_\nu,\nu_x^\infty)$ the \emph{(generalized) Young measure} generated by the (subsequence of the) $v_j$'s.

In plasticity theory~\cite{Suquet78,Suquet79,TemamStrang80}, one often deals with sequences of uniformly $\Lrm^1$-bounded \emph{symmetric gradients}
\[
  \Ecal u_j := \frac{1}{2} \bigl( \nabla u_j + \nabla u_j^T \bigr).
\]
It is an important problem to characterize the (generalized) Young measures $\nu$ generated by such sequences $(\Ecal u_j)$. We call such $\nu$ \emph{BD-Young measures} and write $\nu \in \BDY(\Omega)$, since all BD-functions~\cite{Suquet78,Suquet79,TemamStrang80} can be reached as weak* limits of sequences $(u_j)$ as above. Recall that a function $u \in \Lrm^1(\Omega;\R^d)$ lies in the space $\BD(\Omega)$ of functions of bounded deformation if its distributional symmetrized derivative $Eu$ is a bounded Radon measure on $\Omega$ taking values in $\Rdds$.

Our main result is the following:

\begin{theorem} \label{thm:BDY_charact}
Let $\nu \in \Ybf(\Omega;\Rdds)$ be a (generalized) Young measure. Then, $\nu$ is a BD-Young measure, $\nu \in \BDY(\Omega)$, if and only if there exists $u \in \BD(\Omega)$ with $[\nu] = Eu$ and for all symmetric-quasiconvex $h \in \Crm(\R_\sym^{d \times d})$ with linear growth at infinity, the Jensen-type inequality
\[
  h \biggl( \dprb{\id,\nu_x} + \dprb{\id,\nu_x^\infty} \frac{\di \lambda_\nu}{\di \Lcal^d}(x) \biggr)
    \leq \dprb{h,\nu_x} + \dprb{h^\#,\nu_x^\infty} \frac{\di \lambda_\nu}{\di \Lcal^d}(x).
\]
holds at $\Lcal^d$-almost every $x \in \Omega$.
\end{theorem}

Here, the \emph{generalized recession function} $h^\# \colon \R^N \to \R$ of a map $h \colon \R^N \to \R$ with \emph{linear growth at infinity}, i.e.\ $\abs{h(A)} \leq C(1+\abs{A})$ for some constant $C > 0$, is given as
\[
  h^\#(A) := \limsup_{\substack{\!\!\!\! A' \to A \\ \; t \to \infty}} \,
  \frac{h(tA')}{t},  \qquad A \in \R^N.
\]
We remark that the use of the generalized recession function can in general not be avoided since not every quasiconvex function with linear growth at infinity has a (strong) recession function (and one needs to test with all those, but see~\cite[Theorem~6.2]{KirchheimKristensen16}). Further, a bounded Borel function $f \colon \R_\sym^{d \times d} \to \R$ is called \emph{symmetric-quasiconvex} if with $\Ecal \psi := (\nabla \psi + \nabla \psi^T)/2$,
\[
  f(A) \leq \dashint_D f(A + \Ecal \psi(y)) \dd y
  \qquad \text{for all $\psi \in \Wrm^{1,\infty}_0(D;\R^d)$ and all $A \in \R_\sym^{d \times d}$.}
\]

For a suitable integrand $f \colon \Omega \times \Rdds \to \R$, the minimum principle
\begin{equation} \label{eq:minprinc_ext}
  \ddpr{f,\nu} \to \min,  \quad\text{$\nu \in \BDY(\Omega)$.}
\end{equation}
can be seen as the \emph{extension-relaxation} of the minimum principle 
\begin{equation} \label{eq:minprinc_orig}
  \int_\Omega f(x,\Ecal u(x)) \dd x + \int_\Omega f^\infty \biggl(x,\frac{\di E^s u}{\di \abs{E^s u}}(x)\biggr) \dd \abs{E^s u} \to \min,
  \quad\text{$u \in \BD(\Omega)$.}
\end{equation}
The point is that~\eqref{eq:minprinc_orig} may not be solvable if $f$ is not symmetric-quasiconvex, whereas~\eqref{eq:minprinc_ext} always has a solution. In this situation, our main Theorem~\ref{thm:BDY_charact} then gives (abstract) restrictions on the Young measures to be considered in~\eqref{eq:minprinc_ext}. Another type of relaxation involving the symmetric-quasiconvex envelope of $f$ is investigated in~\cite{ArroyoRabasaDePhilippisRindler17?} within the framework of general linear PDE side-constraints.

Our proof of Theorem~\ref{thm:BDY_charact} roughly follows the \enquote{local} strategy developed in~\cite{Rindler14YM} for the characterization of BV-Young measures. The necessity part follows from a lower semicontinuity theorem, in this case the BD-lower semicontinuity result from~\cite{Rindler11}, as usual. For the sufficiency part, we first characterize \enquote{special} Young measures that can be generated by sequences in BD, see Section~\ref{sc:local}. These \enquote{special} Young measures originate from a blow-up procedure and are called \emph{tangent Young measures}. There are two types: regular and singular tangent Young measures, depending on whether regular (Lebesgue measure-like) effects or singular effects dominate around the blow-up point.

For the regular tangent Young measures the classical methods of Kinderlehrer \& Pedregal~\cite{KinderlehrerPedregal91,KinderlehrerPedregal94} are applicable. In order to deal with singular tangent measures, we first need to strengthen the result on \enquote{good blow-ups} for Young measures with a BD-barycenter from~\cite{Rindler11}, see Lemma~\ref{lem:very_good_blowups}, which is also interesting in its own right. We combine this lemma with the analogue of Alberti's rank-one theorem in BD from~\cite{DePhilippisRindler16}, which imposes strong constraints on the underlying BD-deformation (discussed in Remark~\ref{rem:BDpolar_necessary}). Glueing tangent BD-Young measures together, see Lemma~\ref{lem:glueing}, we then prove Theorem~\ref{thm:BDY_charact} in Section~\ref{sc:proof}.

We stress that our argument crucially rests on the BD-analogue of Alberti's rank-one theorem recently proved by the authors in~\cite{DePhilippisRindler16}, see Theorem~\ref{thm:BDpolar}. The reason is that this result explains the local structure of singularities that can occur in BD-functions (more precisely, in the singular part of the symmetric derivative). A weaker version of this argument was already pivotal in the work~\cite{Rindler11}. However, to prove Theorem~\ref{thm:BDY_charact}, the strong version of~\cite{DePhilippisRindler16} is needed. Technically, in one of the proof steps to establish Theorem~\ref{thm:BDY_charact} we need to create \enquote{artificial concentrations} by compressing symmetric gradients in one direction. This is only possible if we know precisely what these singularities look like, see Lemma~\ref{lem:artconc} and also Remark~\ref{rem:BDpolar_necessary} for details. 
It is not clear to us if the use of Theorem~\ref{thm:BDpolar} can be avoided to prove Theorem~\ref{thm:BDY_charact}

As another very useful technical tool, we utilize the BD-analogue of the surprising observation by Kirchheim \& Kristensen~\cite{KirchheimKristensen16} that the singular part of a BV-Young measure is \emph{unconstrained}. Without this observation, a weaker characterization result could be established, where, however  a second, \emph{singular} Jensen-type inequality needs to be required. Indeed, it was shown in Theorem~4 of~\cite{Rindler11} that in the situation of our theorem automatically also the \emph{singular} Jensen inequality 
\[
  h^\# \bigl( \dprb{\id,\nu_x^\infty} \bigr) \leq \dprb{h^\#,\nu_x^\infty}  \qquad
  \text{for $\lambda_\nu^s$-almost every $x \in \Omega$}
\]
holds. It is remarkable (and due to the observations in~\cite{KirchheimKristensen16} alluded to above) that this is, however, not needed to prove the characterization result.

The third central new ingredient in the proof is an argument yielding \enquote{very good} blow-ups at singular points, which are not only two-dimensional, but even one-dimensional (plus an affine part). This is achieved by iterating the blow-up construction (using the observation that  \enquote{blow-ups of blow-ups are blow-ups}); see Lemma~\ref{lem:very_good_blowups} for details.

As a (technical) application of Theorem~\ref{thm:BDY_charact}, we show how our result can be used to split off an atomic part from a BD-Young measure in generating sequences, see Theorem~\ref{thm:split}.

\subsection*{Acknowledgments} 

We would  like to warmly thank the referees for their careful reading of the manuscript and many helpful suggestions, which have led to the improvement of the manuscript. G.~D.~P. is supported by the MIUR SIR-grant ``Geometric Variational Problems" (RBSI14RVEZ) and F.~R.\ gratefully acknowledges the support from an EPSRC Research Fellowship on ``Singularities in Nonlinear PDEs'' (EP/L018934/1).

\section{Setup and preliminary results} \label{sc:setup}

In this section we recall all the notation and technical tools that will be employed in the subsequent sections. In particular, we collect many results from the framework of generalized Young measure, usually specialized to the BD-case.

\subsection{Functions of bounded deformation}

The space BD was introduced in~\cite{MatthiesStrangChristiansen79book,Suquet78,Suquet79,TemamStrang80} for applications in plasticity theory, much of the theory relevant to this work is developed in~\cite{TemamStrang80,Kohn82,Temam85book,AmbrosioCosciaDalMaso97,Babadjian15,ContiFocardiIurlano15}.

As a standing assumption throughout this whole work, let $\Omega \subset \R^d$ be an open domain with Lipschitz boundary; in the following proofs we implicitly assume $d \geq 2$, but the main results are (trivially)  true also for $d = 1$ since then BV and BD agree and (symmetric-)quasiconvexity is just convexity. The space $\BD(\Omega)$ of \term{functions of bounded deformation} is defined as the space of functions $u \in \Lrm^1(\Omega;\R^d)$ such that the distributional \term{symmetric derivative}
\[
  Eu := \frac{Du + Du^T}{2}
\]
is (representable as) a finite Radon measure $Eu \in \Mcal(\Omega;\R_\sym^{d \times d})$. Clearly, $\BD(\Omega)$ is a Banach space under the norm $\norm{u}_{\BD(\Omega)} := \norm{u}_{\Lrm^1(\Omega;\R^d)} + \abs{Eu}(\Omega)$.

We split $Eu$ according to the Lebesgue--Radon--Nikod\'{y}m decomposition as
\[
  Eu = \Ecal u \, \Lcal^d + E^s u,  \qquad
  \Ecal u := \frac{\di Eu}{\di \Lcal^d} \in \Lrm^1(\Omega;\R_\sym^{d \times d}),
\]
where the \term{approximate symmetrized gradient} $\Ecal u$ is the Radon--Nikod\'{y}m derivative of $Eu$ with respect to Lebesgue measure and $E^s u$ is the \term{singular part} of $Eu$ (with respect to $\Lcal^d$).

Since there is no Korn inequality in $\Lrm^1$, see~\cite{Ornstein62,ContiFaracoMaggi05,KirchheimKristensen16}, it can be shown that $\BV(\Omega;\R^d)$ is a proper subspace of $\BD(\Omega)$. See~\cite{ContiFocardiIurlano15} for further results in this direction.

A \term{rigid deformation} is a skew-symmetric affine map $r \colon \R^d \to \R^d$, i.e.\ $u$ is of the form
\[
  r(x) = u_0 + Rx,  \qquad \text{where $u_0 \in \R^d$, $R \in \R_\skw^{d \times d}$.}
\]
We have the following Poincar\'{e} inequality: For each $u \in \BD(\Omega)$ there exists a rigid deformation $r$ such that
\begin{equation} \label{eq:Poincare_ext}
  \norm{u+r}_{\Lrm^{d/(d-1)}(\Omega;\R^d)} \leq C \abs{Eu}(\Omega),
\end{equation}
where $C = C(\Omega)$ only depends on the domain $\Omega$. This is shown for example in~\cite{TemamStrang80} or see~\cite[Remark~II.2.5]{Temam85book}.

Finally, we will also define the \term{symmetric tensor product} $a \odot b := (a \otimes b + b \otimes a)/2 = (ab^T + ba^T)/2$ of two vectors $a,b \in \R^d$.

\subsection{Symmetric-quasiconvexity}

The appropriate generalized convexity notion related to symmetrized gradients is the following: We call a bounded Borel function $f \colon \R_\sym^{d \times d} \to \R$ \term{symmetric-quasiconvex} if
\[
  f(A) \leq \dashint_D f(A + \Ecal \psi(y)) \dd y
  \qquad \text{for all $\psi \in \Wrm^{1,\infty}_0(D;\R^d)$ and all $A \in \R_\sym^{d \times d}$,}
\]
where $D \subset \R^d$ is any bounded Lipschitz domain. Similar assertions to the ones for quasiconvex functions hold, cf.~\cite{Ebobisse00} and~\cite{BarrosoFonsecaToader00}. In particular, if $f$ has linear growth at infinity, we may replace the space $\Wrm_0^{1,\infty}(D;\R^d)$ in the above formula by $\Wrm_0^{1,1}(D;\R^d)$. It can further be shown, see~\cite[Proposition~3.4]{FonsecaMuller99} that any symmetric-quasiconvex $f$ is \emph{convex} in the directions $\R(a \odot b)$ for any $a, b \in \R^d \setminus \{0\}$.

The \term{symmetric-quasiconvex envelope} $SQf \colon \Rdds \to \R$ of a Borel function $f \colon \Rdds \to \R$ is
\[
  SQf(A) = \sup \setb{ g(A) }{ \text{$g$ symmetric-quasiconvex and $g \leq f$} }.
\]
This expression is either identically $-\infty$ or finite and symmetric-quasiconvex. Analogously to the case for usual quasiconvexity (cf.~\cite{Dacorogna08book,KinderlehrerPedregal91}), for continuous $f$, the symmetric-quasiconvex envelope can be written as
\[
  SQf(A) = \inf \setBB{ \dashint_D f \bigl( A + \Ecal \psi(z) \bigr) \dd z }
  { \psi \in \Wrm_0^{1,\infty}(D;\R^d) }.
\]

\subsection{Generalized Young measures}

The following theory is mostly from~\cite{AlibertBouchitte97,KristensenRindler10YM,Rindler11}, where also proofs and examples can be found.

Let again $\Omega \subset \R^d$ be a bounded Lipschitz domain. For $f \in \Crm(\cl{\Omega} \times \R^N)$ and $g \in \Crm(\cl{\Omega} \times \Bbb^N)$, where $\Bbb^N$ denotes the open unit ball in $\R^N$, we let
\begin{align}
    (Rf)(x,\hat{A}) &:= (1-\absn{\hat{A}}) \, f \biggl(x, \frac{\hat{A}}{1-\absn{\hat{A}}} \biggr),
      \qquad x \in \cl{\Omega}, \, \hat{A} \in \Bbb^N, \quad \text{and}  \label{eq:R} \\
    (R^{-1}g)(x,A) &:= (1+\absn{A}) \, g \biggl(x, \frac{A}{1+\absn{A}} \biggr),
      \qquad x \in \cl{\Omega}, \, A \in \R^N. \notag
\end{align}
Clearly, $R^{-1}Rf = f$ and $RR^{-1}g = g$. Define
\[
  \Ebf(\Omega;\R^N) := \setb{ f \in \Crm(\cl{\Omega} \times \R^N) }{ \text{$Rf$ extends continuously onto $\cl{\Omega \times \Bbb^N}$} }.
\]
In particular, $f \in \Ebf(\Omega;\R^N)$ has \term{linear growth at infinity}, i.e.\ there exists a constant $M \geq 0$ (in fact, $M = \norm{Rf}_{\Lrm^\infty(\Omega \times \Bbb^N)}$) with
\[
\abs{f(x,A)} \leq M(1+\abs{A})  \qquad \text{for all } x \in \cl{\Omega}, \, A \in \R^N.
\]
Furthermore, for all $f \in \Ebf(\Omega;\R^N)$, the \term{(strong) recession function} $f^\infty \colon \cl{\Omega} \times \R^N \to \R$, defined as
\begin{equation} \label{eq:f_infty}
  f^\infty(x,A) := \lim_{\substack{\!\!\!\! x' \to x \\ \!\!\!\! A' \to A \\ \; t \to \infty}}
    \frac{f(x',tA')}{t},  \qquad x \in \cl{\Omega}, \, A \in \R^N,
\end{equation}
exists and takes finite values. Clearly, $f^\infty$ is positively $1$-homogeneous in $A$, that is $f^\infty(x,\alpha A) = \alpha f^\infty(x,A)$ for all $\alpha \geq 0$. It can be shown that in fact $f \in \Crm(\cl{\Omega};\R^N)$ is in the class $\Ebf(\Omega;\R^N)$ if and only if  $f^\infty$ exists in the sense~\eqref{eq:f_infty}. 

For $f \in \Crm(\cl{\Omega} \times \R^N)$ with linear growth at infinity, $f^\infty$ may not exist, but we can always define the \term{generalized recession function} $f^\# \colon \cl{\Omega} \times \R^N \to \R$ via
\[
  f^\#(x,A) := \limsup_{\substack{\!\!\!\! x' \to x \\ \!\!\!\! A' \to A \\ \; t \to \infty}} \,
  \frac{f(x',tA')}{t},  \qquad x \in \cl{\Omega}, \, A \in \R^N.
\]
It is easy to see that $f^\#$ is always positively $1$-homogeneous and upper semicontinuous in its second argument. In many other works, $f^\#$ is just called the \enquote{recession function}, but here the distinction to our (strong) recession function $f^\infty$ is important.

A \term{(generalized) Young measure} $\nu \in \Ybf(\Omega;\R^N) \subset \Ebf(\Omega;\R^N)^*$ on the open set $\Omega \subset \R^d$ with values in $\R^N$ is a triple $\nu = (\nu_x,\lambda_\nu,\nu_x^\infty)$ consisting of
\begin{itemize}
  \item[(i)] a parametrized family of probability measures $(\nu_x)_{x \in \Omega} \subset \Mcal_1(\R^N)$, called the \term{oscillation measure};
  \item[(ii)] a positive finite measure $\lambda_\nu \in \Mcal_+(\cl{\Omega})$, called the \term{concentration measure}; and
  \item[(iii)] a parametrized family of probability measures $(\nu_x^\infty)_{x \in \cl{\Omega}} \subset \Mcal_1(\Sbb^{N-1})$, called the \term{concentration-direction measure},
\end{itemize}
for which we require that
\begin{itemize}
  \item[(iv)] the map $x \mapsto \nu_x$ is weakly* measurable with respect to $\Lcal^d$, i.e.\ the function $x \mapsto \dpr{f(x,\frarg),\nu_x}$ is $\Lcal^d$-measurable for all bounded Borel functions $f \colon \Omega \times \R^N \to \R$,
  \item[(v)] the map $x \mapsto \nu_x^\infty$ is weakly* measurable with respect to $\lambda_\nu$, and
  \item[(vi)] $x \mapsto \dprn{\abs{\frarg},\nu_x} \in \Lrm^1(\Omega)$.
\end{itemize}
Equivalently to (i)--(vi), one may require
\begin{align*}
  &(\nu_x) \in \Lrm_{w*}^\infty(\Omega;\Mcal_1(\R^N)),
    &&\lambda_\nu \in \Mcal_+(\cl{\Omega}), & \\
  &(\nu_x^\infty) \in \Lrm_{w*}^\infty(\cl{\Omega},\lambda_\nu;\Mcal_1(\Sbb^{N-1})),
    &&x \mapsto \dprb{\abs{\frarg},\nu_x} \in \Lrm^1(\Omega).
\end{align*}

The \term{duality pairing} between $f \in \Ebf(\Omega;\R^N)$ and $\nu \in \Ybf(\Omega;\R^N)$ is given as
\begin{align*}
  \ddprb{f,\nu} &:= \int_\Omega \dprb{f(x,\frarg), \nu_x} \dd x
    + \int_{\cl{\Omega}} \dprb{f^\infty(x,\frarg),\nu_x^\infty} \dd \lambda_\nu(x) \\
  &:= \int_\Omega \int_{\R^N} f(x,A) \dd \nu_x(A) \dd x
  + \int_{\cl{\Omega}} \int_{\partial \Bbb^N} f^\infty(x,A) \dd \nu_x^\infty(A) \dd \lambda_\nu(x).
\end{align*}
The weak* convergence $\nu_j \toweakstar \nu$ in $\Ybf(\Omega;\R^N) \subset \Ebf(\Omega;\R^N)^*$ is then defined with respect to this duality pairing. If $(\gamma_j) \subset \Mcal(\cl{\Omega};\R^N)$ is a sequence of measures with $\sup_j \abs{\gamma_j}(\cl{\Omega}) < \infty$, then we say that the sequence $(\gamma_j)$ \term{generates} a Young measure $\nu \in \Ybf(\Omega;\R^N)$, in symbols $\gamma_j \toY \nu$, if for all $f \in \Ebf(\Omega;\R^N)$ it holds that
\begin{align*}
&f \biggl( x, \frac{\di \gamma_j}{\di \Lcal^d}(x)\biggr) \,\Lcal^d \restrict \Omega
+ f^\infty \biggl(x, \frac{\di \gamma^s_j}{\di \abs{\gamma^s_j}}(x) \biggr) \, \abs{\gamma^s_j}\\
&\qquad\toweakstar\;\; \dprb{f(x,\frarg), \nu_x} \, \Lcal^d \restrict \Omega + \dprb{f^\infty(x,\frarg),
\nu_x^\infty} \, \lambda_\nu  \qquad\text{in $\Mcal(\cl{\Omega})$.}
\end{align*}
Here, $\gamma_j^s$ is the singular part of $\gamma_j$ with respect to Lebesgue measure. Equivalently, we could have defined $\gamma_j \toY \nu$ by requiring that $\delta_{\gamma_j} \toweakstar \nu$, where $\delta_{\gamma_j}$ are \enquote{elementary Young measures} that are naturally associated with the $\gamma_j$.

Also, for $\nu \in \Ybf(\Omega;\R^N)$ we define the \term{barycenter} as the measure
\[
  [\nu] := \dprb{\id, \nu_x} \, \Lcal^d \restrict \Omega + \dprb{\id, \nu_x^\infty} \, \lambda_\nu
  \in \Mcal(\cl{\Omega};\R^N).
\]

The following is the central compactness result in $\Ybf(\Omega;\R^N)$:

\begin{lemma}[Compactness] \label{lem:compact}
Let $(\nu_j) \subset \Ybf(\Omega;\R^N)$ be such that
\[
  \supmod_j \, \ddprb{\ONE \otimes \abs{\frarg}, \nu_j} < \infty.
\]
Then, $(\nu_j)$ is weakly* sequentially relatively compact in $\Ybf(\Omega;\R^N)$, i.e.\ there exists a subsequence (not relabeled) such that $\nu_j \toweakstar \nu$ and $\nu \in \Ybf(\Omega;\R^N)$.
\end{lemma}

In particular, if $(\gamma_j) \subset \Mcal(\cl{\Omega};\R^N)$ is a sequence of measures with $\sup_j \abs{\gamma_j}(\cl{\Omega}) < \infty$ as above, then there exists a subsequence (not relabeled) and $\nu \in \Ybf(\Omega;\R^N)$ such that $\gamma_j \toY \nu$.

By a standard density argument it suffices to check weak*-convergence of Young measures by testing with a countable set of integrands only, which is equivalent to the \emph{separability} of the space $\Ebf(\Omega;\R^N)$:
 
\begin{lemma} \label{lem:dense}
There exists a countable family $\{\phi_\ell \otimes h_\ell\}_{\ell \in \N} \subset \Ebf(\Omega;\R^N)$, where $\phi_\ell \in \Crm(\cl{\Omega})$ and $h_\ell \in \Crm(\R^N)$ such that for $\nu_1,\nu_2 \in \Ybf(\Omega;\R^N)$ the following implication holds:
\[
  \text{$\ddprb{\phi_\ell \otimes h_\ell, \nu_1} = \ddprb{\phi_\ell \otimes h_\ell, \nu_2}$ for all $\ell \in \N$}\quad\Longrightarrow\quad \nu_1 = \nu_2.
\]
Moreover, all $h_\ell$ can be chosen Lipschitz continuous and each $h_\ell$ has either compact support or is positively $1$-homogeneous.
\end{lemma}

\subsection{BD-Young measures}

A Young measure in $\Ybf(\Omega;\R_\sym^{d \times d})$ is called a \term{BD-Young measure}, $\nu \in \BDY(\Omega)$, if it can be generated by a sequence of BD-symmetric derivatives. That is, for all $\nu \in \BDY(\Omega)$, there exists a (necessarily norm-bounded) sequence $(u_j) \subset \BD(\Omega)$ with $Eu_j \toY \nu$. When working with $\BDY(\Omega)$, the appropriate space of integrands is $\Ebf(\Omega;\R_\sym^{d \times d})$, since it is clear that both $\nu_x$ and $\nu_x^\infty$ only take values in $\R_\sym^{d \times d}$ whenever $\nu \in \BDY(\Omega)$. It is easy to see that for a BD-Young measure $\nu \in \BDY(\Omega)$ there exists $u \in \BD(\Omega)$ satisfying $Eu = [\nu] \restrict \Omega$; any such $u$ is called an \term{underlying deformation} of $\nu$. 

The following results about BD-Young measures constitute a \enquote{calculus} for BD-Young measures, which will be used frequently in the sequel see~\cite{KristensenRindler10YM,Rindler11,Rindler11PhD} for proofs (the first reference treats BV-Young measures, but the proofs adapt line-by-line).

\begin{lemma}[Good generating sequences] \label{lem:good_genseq}
Let $\nu \in \BDY(\Omega)$.
\begin{itemize}
  \item[(i)] There exists a generating sequence $(u_j) \subset \BD(\Omega) \cap \Crm^\infty(\Omega;\R^d)$ with $Eu_j \toY \nu$.
  \item[(ii)] If additionally $\lambda_\nu(\partial \Omega) = 0$, then the $u_j$ from (i) can be chosen to satisfy $u_j|_{\partial \Omega} = u|_{\partial \Omega}$, where $u \in \BD(\Omega)$ is any underlying deformation of $\nu$.
\end{itemize}
\end{lemma}

The proof of this result can be found in~\cite[Lemma~4]{KristensenRindler10YM}.

\begin{lemma}[Averaging] \label{lem:averaging}
Let $\nu \in \BDY(\Omega)$ satisfy $\lambda_\nu(\partial \Omega) = 0$. Also, assume $[\nu] = Eu$ for some $u \in \BD(\Omega)$ satisfying on
e of the following two properties:
\begin{itemize}
  \item[(i)] $u$ agrees with an affine map on the boundary $\partial \Omega$ or
  \item[(ii)] $\Omega$ is a cuboid with one face normal $\xi \in \Sbb^{d-1} = \partial \Bbb^d$ and $u$ is $\xi$-directional, that is $u(x) = \eta h(x \cdot \xi)$ with $h \in \BV(\R)$ for some $\eta \in \R^d$.
\end{itemize}
Then, there exists a Young measure $\bar{\nu} \in \BDY(\Omega)$ acting on $f \in \Ebf(\Omega;\Rdds)$ as
\begin{align}
  \ddprb{f,\bar{\nu}} &= \int_\Omega \dashint_\Omega \dprb{f(x,\frarg),\nu_y} \dd y \dd x  \notag\\
  &\qquad + \int_\Omega \dashint_{\cl{\Omega}}
  \dprb{f^\infty(x,\frarg),\nu_y^\infty} \dd \lambda_\nu(y) \cdot  \frac{\lambda_\nu(\cl{\Omega})}{\abs{\Omega}} \dd x.   \label{eq:averaged_GYM_action}
\end{align}
More precisely:
\begin{enumerate}
  \item[(1)] The oscillation measure $(\bar{\nu}_x)_x$ is $\Lcal^d$-essentially constant in $x$ and for all $h \in \Crm(\Rdds)$ with linear growth at infinity it holds that
\[
  \qquad\dprb{h,\bar{\nu}_x} = \dashint_\Omega \dprb{h, \nu_y} \dd y  \qquad\text{a.e.}
\]
  \item[(2)] The concentration measure $\lambda_{\bar{\nu}}$ is a multiple of Lebesgue measure, $\lambda_{\bar{\nu}} = \alpha \Lcal^d \restrict \Omega$, where $\alpha = \lambda_\nu(\cl{\Omega}) / \abs{\Omega}$.
  \item[(3)] The concentration-direction measure $(\bar{\nu}_x^\infty)_x$ is $\Lcal^d$-essentially ($\lambda_{\bar{\nu}}$-essentially) constant and for all $h^\infty \in \Crm(\partial \Bbb^{d \times d}_\sym)$ it holds that
\[
  \qquad\dprb{h^\infty,\bar{\nu}_x^\infty} = \dashint_{\cl{\Omega}} \dprb{h^\infty,\nu_y^\infty}
    \dd \lambda_\nu(y)   \qquad\text{a.e.}
\]
\end{enumerate}
\end{lemma}

\begin{remark} \label{rem:homYM_domain}
We remark that one may consider any averaged Young measure as in the preceding lemma to be defined on \emph{any} bounded Lipschitz domain $D \subset \R^d$, so that $\bar{\nu} \in \BDY(D)$ and~\eqref{eq:averaged_GYM_action} is replaced by
\begin{align*}
  \ddprb{f,\bar{\nu}} &= \int_D \dashint_\Omega \dprb{f(x,\frarg),\nu_y} \dd y \dd x  \notag\\
    &\qquad + \int_D \dashint_{\cl{\Omega}}
    \dprb{f^\infty(x,\frarg),\nu_y^\infty} \dd \lambda_\nu(y) \cdot  \frac{\lambda_\nu(\cl{\Omega})}{\abs{\Omega}} \dd x
\end{align*}
for any $f \in \Ebf(D;\Rdds)$. This can be achieved by a covering argument analogous to the proof of Lemma~\ref{lem:averaging} in ~\cite[Proposition~7]{KristensenRindler10YM} (covering $D$ with rescaled copies of $\Omega$).
\end{remark}

The proof of case~(i) is contained in~\cite[Proposition~7]{KristensenRindler10YM}, the proof of~(ii) is similar, but requires an additional standard staircase (piecewise affine) construction to glue the rescaled versions of $u$ together without incurring an additional jump part: Let $\Omega$ and $(u_j)$ be as in~(ii). First, by Lemma~\ref{lem:good_genseq} we may assume that there exists a sequence $(u_j) \subset \BD(\Omega) \cap \Crm^\infty(\Omega;\R^d)$ with $Eu_j \toY \nu$ and $u_j|_{\partial \Omega} = u|_{\partial \Omega}$. For every $j \in \N$ let $a_{jk} \in \R^d$ be defined such that the similar rescaled sets $\Omega_{jkl} := a_{jkl} + j^{-1} \Omega$, $k = 1,\ldots,j$, $l = 1,\ldots,j^{d-1}$, form a cover of $\Omega$. We furthermore assume that the $\Omega_{jkl}$ are arranged in a regular grid with $\Omega_{jkl}$ ($l = 1,\ldots,j^{d-1}$) lying in the $k$'th slice in $\xi$-direction.

Furthermore, denote by $\gamma \in \R^d$ the difference between the trace of $u$ on the face of $u$ in the positive $\xi$-direction and $u$ in the negative $\xi$-direction; note that because of the assumption that $u$ has the shape $u(x) = \eta h(x \cdot \xi)$, $\gamma$ is a \emph{constant vector}. Then define
\[
  w_j(x) := \begin{cases}
    \frac{1}{j} u_j \bigl( j(x-a_{jkl}) \bigr) + \gamma \dfrac{k}{j}
       & \text{if } x \in a_{jkl} + j^{-1} \Omega, \\
    0  & \text{otherwise.}
  \end{cases}
\]
It is easy to see that $w_j \in \Wrm^{1,1}(\Omega;\R^m)$ (recall that $u_j|_{\partial \Omega} = u|_{\partial \Omega}$). For the weak derivative we get
\[
  \nabla w_j(x) = \begin{cases}
    \nabla u_j \bigl( j(x-a_{jk}) \bigr)  & \text{if }
      x \in a_{jk} + j^{-1} \Omega, \\
    0  & \text{otherwise.}
  \end{cases}
\]
Note that the staircase term $\gamma k/j$ is chosen precisely to annihilate the jumps over the slice boundaries in direction $\xi$; over the other boundaries there are no jumps by the assumption on the shape of $u$. It can now be checked, following line-by-line the proof of~\cite[Proposition~7]{KristensenRindler10YM}, that the $w_j$ generate $\bar{\nu}$ as required in the lemma.

A special case is the following corollary:

\begin{corollary}[Generalized Riemann--Lebesgue lemma] \label{cor:Riemann_Lebesgue}
Let $u \in \BD(\Omega)$ that satisfied~(i) or~(ii) from the previous lemma. Then, for every open bounded Lipschitz domain $D \subset \R^d$ there exists $\nu \in \BDY(D)$ that acts on $f \in \Ebf(D;\Rdds)$ as
\begin{align*}
  \ddprb{f,\nu} &= \int_D \dashint_\Omega f(x,\Ecal u(y)) \dd y \dd x \\
  &\qquad + \frac{\abs{E^s u}(\Omega)}{\abs{\Omega}} \int_D \dashint_\Omega f^\infty \biggl(x,
    \frac{\di E^s u}{\di \abs{E^s u}}(y) \biggr) \dd \abs{E^s u}(y) \dd x.
\end{align*}
Moreover, $\lambda_\nu(\partial \Omega) = 0$.
\end{corollary}

We will also need the following approximation result, see~\cite[Proposition~8]{KristensenRindler10YM}.

\begin{lemma}[Approximation] \label{lem:approx}
Let $\nu \in \BDY(\Omega)$ satisfy $\lambda_\nu(\partial \Omega) = 0$. Also, assume that $[\nu] = Eu$ for $u \in \BD(\Omega)$ satisfying one of the conditions~(i),~(ii) from Lemma~\ref{lem:averaging}. Then, for all $k \in \N$, there exists a partition $(C_{kl})_l$ of $(\Lcal^d + \lambda_\nu)$-almost all of $\cl{\Omega}$ into open sets $C_{kl}$, $l = 1,\ldots,N(k)$, with diameters at most $1/k$ and $(\Lcal^d + \lambda_\nu)(\partial C_{kl}) = 0$, and a sequence of Young measures $(\nu_k) \subset \BDY(\Omega)$ such that
\[
  \nu_k \toweakstar \nu  \quad\text{in $\Ybf(\Omega;\Rdds)$}
  \qquad\text{as $k \to \infty$}
\]
and for every $f \in \Ebf(\Omega;\Rdds)$
\[
  \ddprb{f,\nu_k} = \sum_{l = 1}^{N(k)} \ddprb{f,\overline{\nu \restrict {C_{kl}}}},
\]
where $\overline{\nu \restrict C_{kl}}$ designates the averaged Young measure (as in Lemma~\ref{lem:averaging}) of the restriction $\nu \restrict C_{kl}$ of $\nu$ to $C_{kl}$.
\end{lemma}

\subsection{Localization of Young measures} \label{ssc:loc}

The paper~\cite{Rindler11} proved two localization principles for BD-Young measures, leading to so-called \enquote{tangent Young measures}; here and in the following for ease of notation we leave out the dependence of the spaces on $\Rdds$.

We first define the following \emph{regular}, or \emph{homogeneous}, spaces of (tangent) Young measures for $A_0 \in \Rdds$ ($Q$ being the standard unit cube).
\begin{align*}
  \Ebf^\reg &:= \setb{ \ONE \otimes h }{ \ONE \otimes h \in \Ebf(Q;\Rdds) }, \\
  \Ybf^\reg(A_0) &:= \setb{ \sigma = (\sigma_x,\lambda_\sigma,\sigma^\infty_x) \in \Ybf(Q;\Rdds) }{ \text{$[\sigma] = A_0 \Lcal^d \restrict Q$,}\\
    &\hspace{30pt} \text{$\sigma_y, \sigma^\infty_y$ constant in $y$, $\lambda_\sigma = \alpha \Lcal^d \restrict Q$ for some $\alpha \geq 0$ } }, \\
  \BDY^\reg(A_0) &:= \Ybf^\reg(A_0) \cap \BDY(Q) \\
    &\phantom{:}= \setb{ \sigma \in \Ybf^\reg(A_0) }{ \text{$\exists (v_j) \subset \BD(Q)$ with $Eu_j \toY \sigma$} }.
\end{align*}

The first localization principle then reads as follows:

\begin{proposition}[Localization at regular points] \label{prop:localize_reg}
Let $\nu \in \Ybf(\Omega;\Rdds)$ be a Young measure. Then, for $\Lcal^d$-almost every $x_0 \in \Omega$ there exists a \term{regular tangent Young measure}
\[
  \sigma \in \Ybf^\reg (A_0),  \qquad \text{where}\qquad
  A_0 := \dprb{\id,\nu_{x_0}} + \dprb{\id,\nu_{x_0}^\infty} \frac{\di \lambda_\nu}{\di \Lcal^d}(x_0),
\]
which satisfies
\begin{align*}
  [\sigma] &= A_0 \, \Lcal^d \restrict Q \in \Tan_Q([\nu],x_0),  &\sigma_y &= \nu_{x_0} \quad\text{a.e.,}  \\
  \lambda_\sigma &= \frac{\di \lambda_\nu}{\di \Lcal^d}(x_0) \, \Lcal^d \restrict Q \in \Tan_Q(\lambda_\nu,x_0),
    &\sigma_y^\infty &= \nu_{x_0}^\infty \quad\text{a.e.}
\end{align*}
Additionally, if $\nu \in \BDY(\Omega)$, then $\sigma \in \BDY^\reg$.
\end{proposition}

Here, $\Tan_Q(\mu,x_0)$ contains all (restricted) \term{tangent measures} of $\mu \in \Mcal(\Omega;\R^N)$ at $x_0 \in \Omega$, i.e.\ those measures $\sigma \in \Mcal(Q;\R^N)$ such that there exists $r_n \todown 0$ and $c_n > 0$ with $c_n T^{x_0,r_n}_\# \mu \toweakstar \sigma$, where $T^{x_0,r_n}(x) := (x-x_0)/r_n$ and 
$$T^{x_0,r_n}_\# \mu := \mu \circ (T^{x_0,r_n})^{-1} = \mu(x_0 + r_n \frarg)$$
 is the push-forward of $\mu$ under $T^{x_0,r_n}$. A proof for the preceding proposition can be found in~\cite[Proposition~1]{Rindler11}.

We furthermore remark that $\sigma$ in Proposition~\ref{prop:localize_reg} is such that for all open sets $U \subset Q$ with $\Lcal^d(\partial U) = 0$, and all $h \in \Crm(\R^{d \times d})$ such that the recession function $h^\infty$ exists in the sense of~\eqref{eq:f_infty}, it holds that
\[
  \ddprb{\ONE_U \otimes h, \sigma} = \biggl[ \dprb{h,\nu_{x_0}} + \dprb{h^\infty,\nu_{x_0}^\infty}
    \frac{\di \lambda_\nu}{\di \Lcal^d}(x_0) \biggr] \abs{U}.
\]

For the singular counterpart to Proposition~\ref{prop:localize_reg}, we first introduce the following spaces for any bounded Lipschitz domain $D \subset \R^d$ (we again omit mention of $\Rdds$ for ease of notation):
\begin{align*}
  \Ebf^\sing(D) &:= \setb{ f \in \Ebf(D;\Rdds) }{ \text{$f(x,\frarg)$ positively $1$-homogeneous} },  \\
  \Ybf^\sing(D) &:= \setb{ \nu = (\nu_x,\lambda_\nu,\nu^\infty_x) }{ \text{$\nu_x = \delta_0$ a.e.} },  \\
  \BDY^\sing(D) &:= \Ybf^\sing(D;\Rdds) \cap \BDY(D).
\end{align*}
The duality pairing between $\Ebf^\sing(D)$ and $\Ybf^\sing(D)$ is
\[
  \ddprb{f,\nu} := \int_{\cl{D}} \int_{\partial \Bbb^{d \times d}_\sym} f(x,A) \dd \nu^\infty_x(A) \dd \lambda_\nu(x).
\]
Furthermore, we say that the sequence of measures $(\mu_j) \subset \Mcal(D;\Rdds)$ generates the singular Young measure $\nu \in \Ybf^\sing(D)$, in symbols $\mu_j \toY \nu$, if
\[
  \int f \biggl( x, \frac{\di \mu_j}{\di \abs{\mu_j}}(x) \biggr) \dd \abs{\mu_j}(x)
  \to  \ddprb{f,\nu}  \qquad
  \text{for all $f \in \Ebf^\sing(D)$.}
\]

\begin{proposition}[Localization at singular points] \label{prop:localize_sing}
Let $\nu \in \Ybf(\Omega;\Rdds)$ be a Young measure. Then, for $\lambda_\nu^s$-almost every $x_0 \in \Omega$ and every bounded Lipschitz domain $D \subset \R^d$, there exists a \term{singular tangent Young measure}
\[
  \sigma \in \Ybf^\sing(D)
\]
satisfying
\begin{align*}
  [\sigma] &\in \Tan_D([\nu],x_0),  &\sigma_y &= \delta_0 \quad\text{a.e.,}  \\
  \lambda_\sigma &\in \Tan_D(\lambda_\nu^s,x_0) \setminus \{0\}, &\sigma_y^\infty &= \nu_{x_0}^\infty
    \quad\text{$\lambda_\sigma$-a.e.}
\end{align*}
Additionally, if $\nu \in \BDY(\Omega)$, then $\sigma \in \BDY^\sing(D)$.
\end{proposition}

A proof of this fact can be found in~\cite[Proposition~2]{Rindler11}.

\subsection{Good singular blow-ups}

In this section, as a preparation for the singular analogue of Proposition~\ref{prop:local_reg}, we establish a result about good blow-ups for BD-Young measures in Lemma~\ref{lem:very_good_blowups} below.

First, we recall from~\cite[Theorem~3]{Rindler11} the following result. We here state it in a slightly different fashion, namely for Young measures with a BD-barycenter instead of \emph{BD-generated} Young measures.

\begin{lemma}[Good singular blow-ups] \label{lem:good_blowups}
Let $\nu \in \Ybf(\Omega)$ be a Young measure with
\[
  [\nu] \restrict \Omega = Eu  \qquad
  \text{for some $u \in \BD(\Omega)$.}
\]
For $\lambda_\nu^s$-almost every $x_0 \in \Omega$ and every bounded Lipschitz domain $D \subset \R^d$, there exists a singular tangent Young measure $\sigma \in \Ybf^\sing(D)$ as in Proposition~\ref{prop:localize_sing} with $[\sigma] = [\sigma] \restrict D = Ev$ for some $v \in \BD(D)$ and the appropriate assertion among the following holds:
\begin{itemize}
  \item[(i)] If $\dprn{\id,\nu_{x_0}^\infty} \notin \setn{a \odot b}{a,b \in \R^d \setminus \{0\}}$ (this includes the case $\dprn{\id,\nu_{x_0}^\infty} = 0$), then $v$ is equal to an affine function almost everywhere.
  \item[(ii)] If $\dprn{\id,\nu_{x_0}^\infty} = a \odot b$ ($a,b \in \R^d \setminus \{0\}$) with $a \neq b$, then there exist functions $g_1,g_2 \in \BV(\R)$, $v_0 \in \R^d$, and a skew-symmetric matrix $R \in \R_\skw^{d \times d}$ such that
\[
  \qquad v(x) = v_0 + g_1(x \cdot a)b + g_2(x \cdot b)a + Rx,  \qquad \text{$x \in \R^d$ a.e.}
\]
  \item[(iii)] If $\dprn{\id,\nu_{x_0}^\infty} = a \odot a$ ($a \in \R^d \setminus \{0\}$), then there exists a function $g \in \BV(\R)$, $v_0 \in \R^d$ and a skew-symmetric matrix $R \in \R_\skw^{d \times d}$ such that
\[
  \qquad v(x) = v_0 + g(x \cdot a)a + Rx,  \qquad \text{$x \in \R^d$ a.e.}
\]
\end{itemize}
\end{lemma}

\begin{remark} \label{rem:good_blowup_remark}
In the preceding theorem, one sees easily that if $\nu \in \BDY(\Omega)$, then also $\sigma \in \BDY^\sing(D)$ (this is the original statement in~\cite[Theorem~3]{Rindler11}). For us, however, this fact is not needed.
\end{remark}

The proof is the same as in~\cite[Theorem~3]{Rindler11}, where the whole argument is only concerned with the barycenter and not the generating sequence. Moreover, we remark that $[\sigma](\partial D)$ can be achieved by a rescaling of the blow-up sequence $r_n \todown 0$ into $\alpha r_n \todown 0$ for some $\alpha \in (0,1)$ such that $[\sigma](\partial (\alpha D)) = 0$ (assuming that $0 \in D$).

The next ingredient we will need is the theorem on the singular structure of BD-functions, proved in~\cite{DePhilippisRindler16}:

\begin{theorem} \label{thm:BDpolar}
Let $u\in \BD(\Omega)$. Then, for $\abs{E^s u}$-almost every $x \in \Omega$, there exist $a(x), b(x)\in \R^d \setminus \{0\}$ such that
\[
  \frac{\di E^s u}{\di \abs{E^s u}}(x) = a(x) \odot b(x).
\]
\end{theorem}

This is the BD-analogue of the following celebrated Alberti rank-one theorem~\cite{Alberti93}: 

\begin{theorem}[Alberti's rank-one theorem] \label{thm:rankone}
Let $u\in \BV(\Omega)$. Then, for $\abs{D^s u}$-almost every $x \in \Omega$, there exist $a(x), b(x)\in \R^d \setminus \{0\}$ such that
\[
  \frac{\di D^s u}{\di \abs{D^s u}}(x) = a(x) \otimes b(x).
\]
\end{theorem}

We will now state and prove a strengthened version of Lemma~\ref{lem:good_blowups}. For this, we first define suitable spaces:

In all of the following, let $\xi \in \Sbb^{d-1}$ and denote by $Q_\xi$ the rotated unit cube ($\abs{Q_\xi} = 1$) with one face normal $\xi$. We first define one-directional versions of the spaces $\Ebf^\sing, \Ybf^\sing, \BDY^\sing$ for $A_0 \in \Rdds \setminus \{0\}$, $\xi \in \Sbb^{d-1}$; as before, we henceforth leave out the dependence of the spaces on $\Rdds$.
\begin{align*}
  \Ebf^\sing(\xi) &:= \setb{ f \in \Ebf^\sing(Q_\xi;\Rdds) }{ f(y,\frarg) = f(y \cdot \xi, \frarg) },  \\
  \Ybf^\sing(A_0,\xi) &:= \setb{ \sigma = (\sigma_x,\lambda_\sigma,\sigma^\infty_x) \in \Ybf^\sing(Q_\xi;\Rdds) }{ \text{$[\sigma] = A_0 \lambda_\sigma$,}\\
    &\hspace{125pt} \text{$\sigma^\infty_y = \sigma^\infty_{y \cdot \xi}$, $\lambda_\sigma$ is $\xi$-directional } }, \\
  \BDY^\sing(A_0,\xi) &:= \Ybf^\sing(A_0,\xi) \cap \BDY^\sing(Q_\xi).
\end{align*}
Here, the $\xi$-directionality of $\lambda_\sigma$ means that for all Borel sets $B \subset Q_\xi$ it holds that $\lambda_\sigma(B + v) = \lambda_\sigma(B)$ for all $v \perp \xi$ such that $B + v \subset Q_\xi$. Notice that we require $A_0 \neq 0$ here (the case $A_0 = 0$ is treated in the next subsection).

Finally, the spaces $\Ybf^\sing_0(A_0,\xi)$, $\BDY^\sing_0(A_0,\xi)$ are defined to incorporate the additional constraint $\lambda_\sigma(\partial \Omega) = 0$.

\begin{lemma}[Very good singular blow-ups] \label{lem:very_good_blowups}
Let $\nu \in \Ybf(\Omega)$ be a Young measure with
\[
  [\nu] \restrict \Omega = Eu  \qquad
  \text{for some $u \in \BD(\Omega)$.}
\]
For $\lambda_\nu^s$-almost every $x_0 \in \Omega$, there exists a singular tangent Young measure
\[
  \sigma \in \Ybf^\sing(\xi \odot \eta, \xi),
\]
and such that $[\sigma] = [\sigma] \restrict Q_\xi = Ev$ for some $v \in \BD(Q_\xi)$ of the form
\[
  v(x) = v_0 + g(x \cdot \xi)\eta + \beta (\xi \otimes \eta)x + Rx,  \qquad \text{$x \in Q_\xi$ a.e.,}
\]
where $v_0 \in \R^d$, $\beta \in \R$, $g \in \BV(\R)$, $\xi, \eta \in \R^d \setminus \{0\}$, and $R \in \R^{d \times d}_\skw$.
\end{lemma}

\begin{remark}
We note in passing that this improvement in fact allows to slightly simplify the proof of the lower semicontinuity result in~\cite{Rindler11} as well.
\end{remark}

\begin{proof}
Let $u \in \BD(\Omega)$ with $[\nu] \restrict \Omega = Eu$. Then, $E^s u = \dpr{\id,\nu_x^\infty} \lambda_\nu^s$ and
\begin{equation} \label{eq:polar}
  \dpr{\id,\nu_x^\infty} = a(x) \odot b(x)   \quad\text{for $\lambda_\nu^s$-a.e.\ $x \in \Omega$.}
\end{equation}
and some $a(x),b(x) \in \R^d$. Indeed, denoting by $\lambda_\nu^*$ the singular part of $\lambda_\nu$ with respect to $\abs{E^s u}$, the above holds at $\abs{E^s u}$-almost every $x \in \Omega$ with $a(x), b(x) \neq 0$ by Theorem~\ref{thm:BDpolar} and at $\lambda_\nu^*$-almost every $x \in \Omega$ with $a(x) = b(x) = 0$.

From the previous Lemma~\ref{lem:good_blowups} we get that there exists a singular tangent Young measure $\tau \in \Ybf^\sing(D)$ to $\nu$ at $\lambda_\nu^s$-almost every $x_0 \in \Omega$ that satisfies~(i),~(ii) or~(iii). We have that for any $w \in \BD(\R^d)$ with $Ew = [\tau]$,
\[
  Ew = [\tau] = \dprn{\id,\nu_{x_0}^\infty} \lambda_\tau = (a \odot b) \lambda_\tau
\]
for some $a,b \in \R^d$. Indeed, by Proposition~\ref{prop:localize_sing} we get $\tau_y^\infty = \nu_{x_0}^\infty$ for $\lambda_\tau$-almost every $y \in D$, which implies the first equality. Further, if $x_0$ is chosen such that~\eqref{eq:polar} holds, we infer $\dprn{\id,\tau_y^\infty} = \dprn{\id,\nu_{x_0}^\infty} = a(x_0) \otimes b(x_0)$ for $\lambda_\tau$-almost every $y \in D$.

\proofstep{Step~1.}
If $a, b$ are parallel, say $a = b$ after rescaling, then the statement of the present lemma follows immediately with $\sigma := \tau$, $\xi := a/\abs{a} = b/\abs{b}$, $D = Q_\xi$ (note that we can choose $D$ in Lemma~\ref{lem:good_blowups} as we like and we can decide beforehand whether $a,b$ are parallel, see~\eqref{eq:polar}). In this case, $\sigma = \tau \in \Ybf^\sing(a \odot a,a)$, as follows directly from Proposition~\ref{prop:localize_sing} and Lemma~\ref{lem:good_blowups}~(iii).

\proofstep{Step~2.}
Only in the case $\dprn{\id,\nu_{x_0}^\infty} = a \odot b$ with $a \neq b$ there is something left to prove. Without loss of generality we assume that $a = \ee_1$ and $b = \ee_2$, which is possible after a change of variables. In this case, by Lemma~\ref{lem:good_blowups} we have that for any $w \in \BD(Q)$ with $Ew = [\tau]$ there exist functions $g_1,g_2 \in \BV(\R)$, $w_0 \in \R^d$, and a skew-symmetric matrix $R \in \R_\skw^{d \times d}$ such that
\[
  w(y) = w_0 + g_1(y^1)\ee_2 + g_2(y^2)\ee_1 + Ry  \qquad
  \text{for a.e.\ $y = (y^1,\ldots,y^d) \in \R^d$.}
\]
Moreover,
\begin{equation} \label{eq:Dw_struct}
  Dw = (\ee_2\otimes \ee_1) Dg_1 \otimes \Lcal^{d-1} + (\ee_1\otimes \ee_2) \Lcal^1 \otimes Dg_2 \otimes \Lcal^{d-2} + R \Lcal^d.
\end{equation}
and $D^s g_1 \otimes \Lcal^{d-1}$ is singular to $\Lcal^1 \otimes D^s g_2 \otimes \Lcal^{d-2}$.

\proofstep{Case~(I).}
If either $Dg_1$ or $Dg_2$ are the zero measure, the conclusion of the theorem is trivially true with $\sigma := \tau$, $\xi := \ee_1$, $\eta := \ee_2$, $D$ the standard open unit cube. So, henceforth assume that both $Dg_1, Dg_2$ are not the zero measure. In the following we denote by $g_1', g_2'$ the \emph{approximate derivatives} of $g_1, g_2$, i.e.\ the densities of $Dg_1, Dg_2$ with respect to Lebesgue measure.

\proofstep{Case~(II).}
If $D^s g_1, D^s g_2 = 0$, then we may use the regular localization principle, Proposition~\ref{prop:localize_reg}, to construct a non-zero regular tangent Young measure $\sigma \in \Ybf^\reg$ of $\tau$ at $y_0 \in \R^d$. We will argue below that in fact $\sigma$ is a singular tangent Young measure to our original $\nu$ at $x_0$ and that $[\sigma] \in \BD(\Omega)$, see Step~3 of the proof.

\proofstep{Case~(III).}
On the other hand, if $D^s g_1 \neq 0$ (without loss of generality), then we claim we can find $y_0 = (s_0,t_0,y^3,\ldots,y^d) \in \R^d$ with the property that there exists a non-zero singular tangent Young measure $\sigma \in \Ybf^\sing(Q)$ to $\tau$ at $y_0$ and that
\[
  \lim_{r \todown 0} \frac{1}{r} \int_{t_0-r}^{t_0+r} \abs{g_2'(t) - \alpha} \dd t = 0
  \qquad\text{and}\qquad
  \lim_{r \todown 0} \frac{1}{r} \abs{D^s g_2}((t_0-r,t_0+r)) = 0
\]
with a constant $\alpha \in \R$.

Indeed, notice that second and third conditions hold for $\Lcal^1$-almost every $t_0$  because the set of Lebesgue points of $g_2'$ has full Lebesgue measure in $\R$ and the Radon--Nikod\'{y}m derivative of $D^s g_2$ by $\Lcal^1$ is zero $\Lcal^1$-almost everywhere. Hence, they hold for $(\abs{D^s g_1} \otimes \Lcal^{d-1})$-almost every $y_0 = (s_0,t_0,y_3,\ldots,y_d) \in \R^d$ by Fubini's theorem. By the singular localization principle for Young measures, Proposition~\ref{prop:localize_sing}, we know that the first property holds for almost every $y \in \R^d$ with respect to $\lambda_\tau^s$. By~\eqref{eq:Dw_struct}, and the fact that $\abs{D^s g_1} \otimes \Lcal^{d-1}$ is singular with respect to $\Lcal^{1}\otimes \abs{D^s g_2} \otimes \Lcal^{d-2}$, we have
\[
  \abs{E^s w} =\abs{\ee_1\odot\ee_2}\big(\abs{D^s g_1} \otimes \Lcal^{d-1}+\Lcal^{1}\otimes \abs{D^s g_2} \otimes \Lcal^{d-2}\big).
\]
Thus, $\abs{D^s g_1} \otimes \Lcal^{d-1} \neq 0$ is absolutely continuous with respect to $\abs{E^s w}$.

Furthermore, $\abs{E^s w}$ is absolutely continuous with respect to $\lambda_\tau^s$ because
\[
  \lambda_\tau^s = \sqrt{2} \abs{\dpr{\id,\nu_{x_0}^\infty}} \lambda_\tau^s = \sqrt{2} \abs{E^s w}
\]
since $\abs{\dpr{\id,\nu_{x_0}^\infty}} = \abs{\ee_1 \odot \ee_2} = 1/\sqrt{2}$. Thus, the first condition also holds at $(\abs{D^s g_1} \otimes \Lcal^{d-1})$-almost every $y_0$ and we find at least one $y_0 = (s_0,t_0,y^3,\ldots,y^d)$ with the claimed properties.

\proofstep{Step~3.}
We observe that $\sigma$ is still a singular tangent Young measure to $\nu$ at the original point $x_0$ if the latter was chosen suitably. Indeed, it can be easily checked that the property of being a singular tangent Young measure is preserved when passing to another \enquote{inner} (regular or singular) tangent Young measure; the only non-obvious fact here is that \enquote{tangent measures of tangent measures are tangent measures}, but this is well-known and proved for example in Theorem~14.16 of~\cite{Mattila95book}.

The \enquote{inner} blow-up sequence of the BD-primitives of the barycenters has the form
\[
  w^{(n)}(z) := r_n^{d-1}c_n w(y_0 + r_n z),  \qquad z \in \R^d,
\]
with $r_n \todown 0$ and constants $c_n = r_n^{-d}$ if we are in case~(II) and
\begin{equation} \label{eq:cn}
  c_n = \frac{1}{\ddpr{\ONE_{Q(y_0,r_n)} \otimes \abs{\frarg}, \tau}}
  = \frac{1}{\sqrt{2} \cdot \lambda_\tau(Q(y_0,r_n))}
\end{equation}
if we are in case~(III). Note that in both cases $\limsup_{n\to\infty} r_n^d c_n < \infty$. To retain a BD-uniformly bounded sequence, it might also be necessary to add an ($n$-dependent) rigid deformation to $w^{(n)}$, but for ease of notation this is omitted above.

Then, $w^{(n)} \toweakstar v$ in $\BD(Q)$ and $v$ has the property that $Ev = [\sigma]$. For $Ew^{(n)}$ we get
\begin{align*}
  Ew^{(n)} &= (\ee_1 \odot \ee_2) \Bigl( r_n^d c_n \bigl[ g_1'(s_0 + r_n z^1)
    + g_2'(t_0 + r_n z^2) \bigr] \, \Lcal^d(\di z) \\
  &\qquad\qquad\qquad + c_n T_\#^{y_0,r_n} \bigl[ D^s g_1 \otimes \Lcal^{d-1}
    + \Lcal^1 \otimes D^s g_2 \otimes \Lcal^{d-2} \bigr] \Bigr).
\end{align*}
Let $\phi \in \Crm_c^\infty(\R^d)$ and choose $R > 0$ so large that $\supp \phi \subset\subset Q(0,R) = (-R,R)^d$. Using the special properties of our choice of $y_0$ together with the fact that $c_n \lesssim r_n^{-d}$ as $n \to \infty$, we observe that
\begin{align*}
  r_n^d c_n \int \phi(z) \, \abs{g_2'(t_0 + r_n z^2) - \alpha} \dd z
    &= c_n \int \phi \Bigl(\frac{y-y_0}{r_n}\Bigr) \, \abs{g_2'(y^2) - \alpha} \dd y \\
  &\leq C \norm{\phi}_\infty \frac{1}{r_n^d} \int_{Q(y_0,r_nR)} \abs{g_2' (y^2) - \alpha} \dd y
\end{align*}
and this goes to zero as $n \to \infty$. Also,
\begin{align*}
  &c_n \int \phi \dd T_\#^{y_0,r_n} \bigl[ \Lcal^1 \otimes D^s g_2 \otimes \Lcal^{d-2} \bigr] \\
  &\qquad = c_n \int\!\!\!\int \phi \Bigl(\frac{y-y_0}{r_n}\Bigr) \dd D^s g_2(y_2) \dd (y_1,y_3,\ldots,y_d) \\
  &\qquad \leq C \norm{\phi}_\infty \frac{\abs{D^s g_2}(t_0+(-r_nR,r_nR))}{r_n} \\
  &\qquad\qquad \cdot \frac{\Lcal^{d-1}((s_0,y_0^3,\ldots,y_0^d) + (-r_nR,r_nR)^{d-1})}{r_n^{d-1}}  \to 0 \qquad\text{as $n \to \infty$.}
\end{align*}
Hence, the $\ee_2$-directional parts of $Ew^{(n)}$ in the limit converge to the \emph{fixed} matrix $\beta (\ee_1 \odot \ee_2)$, where $\beta = \alpha \lim_{n\to\infty} r_n^d c_n$ (this limit exists after taking a subsequence). More precisely, in case~(II), $_\# r_n^{-d}$ and $\sigma$ is a regular tangent Young measure to $\tau$, thus potentially $\beta \neq 0$; otherwise, in case~(III) it must hold that $\beta = 0$ (since $c_n \sim \lambda_\tau^s(Q(y_0,r_n))^{-1}$). The $\ee_1$-directional parts of $Ew^{(n)}$ clearly stay $\ee_1$-directional under the operation of taking weak* limits. Thus, $v$ is of the required form, with $(\xi,\eta) = (\ee_1,\ee_2)$.
\end{proof}

\subsection{Functional analytic properties of $\BDY^\reg, \BDY^\sing$}

In this section we prove the following lemma about our tangent Young measure spaces:

\begin{lemma} \label{lem:BDYregsing_convex}
The sets $\BDY^\reg$ and $\BDY^\sing(a \odot b,\xi)$ are convex and weakly* closed (with respect to the topology induced as a subset of $(\Ebf^\reg)^*$ and $\Ebf^\sing(\xi)^*$) for all $a,b \in \R^d \setminus \{0\}$, $\xi \in \{a,b\}$.
\end{lemma}

\begin{proof}
We only prove the statements for $\BDY^\sing(a \odot b,\xi)$ since they are more difficult (for regular BD-Young measures the argument is easier because our underlying deformation of the homogeneous Young measure is even affine). We furthermore assume without loss of generality that $\xi = a$.

\proofstep{Step~1: Weak*-closedness of $\BDY^\sing(a \odot b,a)$.}
Let $\{f_n\}_{n \in \N} = \{\phi_n \otimes h_n\}_{n \in \N} \subset \Ebf^\sing(a)$ be a countable set of integrands that determines Young measure convergence in $\Ybf^\sing(a \odot b,a)$, this can be achieved by a reasoning analogous to Lemma~\ref{lem:dense} (only take $\xi$-directional $\phi_n$ and positively $1$-homogeneous $h_n$). Let $\sigma$ be in the weak* closure of $\BDY^\sing(a \odot b,a)$. Then, for every $j \in \N$ there exists $\sigma_j \in \BDY^\sing(a \odot b,a)$ with
\[
  \absb{\ddprb{f_n, \sigma_j} - \ddprb{f_n, \sigma}}
  + \absb{\ddprb{\ONE \otimes \abs{\frarg}, \sigma_j} - \ddprb{\ONE \otimes \abs{\frarg}, \sigma}}
  \leq \frac{1}{j}  \qquad
  \text{for all $n \leq j$.}
\]
In particular, $\sigma_j \toweakstar \sigma$ in $\Ebf^\sing(a)^*$ and also in $\Ybf(Q_a;\Rdds)$ since the sequence $(\sigma_j)$ is compact in that space. Indeed, $\ddpr{\ONE \otimes \abs{\frarg}, \sigma_j}$ is uniformly in $j$ bounded, so we may use the compactness result from Lemma~\ref{lem:compact}. It is not hard to check that the defining properties of $\Ybf^\sing(a \odot b,a)$ (such as the $a$-directionality of $\lambda_\sigma, y \mapsto \nu^\infty_y$) are preserved under weak* limits, hence $\sigma \in \Ybf^\sing(a \odot b,a)$.

We need to show that also $\sigma$ is generated by a sequence of symmetric gradients, which follows by a diagonal argument: Select for each $j \in \N$ a function $u_j \in \BD(Q_a) \cap \Crm^\infty(Q_a;\R^d)$ (see Lemma~\ref{lem:good_genseq}) with the property that
\[
  \absBB{\int_{Q_a} f_n(x,\Ecal u_j(x)) \dd x - \ddprb{f_n,\sigma_j}}
  + \absb{\norm{\Ecal u_j}_{\Lrm^1} - \ddprb{\ONE \otimes \abs{\frarg},\sigma_j}}
  \leq \frac{1}{j}
\]
for all $n \leq j$. Then, adding a rigid deformation to the $u_j$'s, Poincar\'{e}'s inequality in $\BD$, see~\eqref{eq:Poincare_ext}, yields that there exists a (non-relabeled) subsequence such that $\Ecal u_j \toY \mu \in \BDY(Q_a)$. Clearly, $\mu = \sigma$ by construction.

\proofstep{Step~2: Convexity of $\BDY^\sing(a \odot b,a)$ assuming $\lambda_\mu(\partial Q_a) = \lambda_\nu(\partial Q_a) = 0$).}
Let $\mu,\nu \in \BDY^\sing(a \odot b,a)$  be such that  $\lambda_\mu(\partial Q_a) = \lambda_\nu(\partial Q_a) = 0$ and let $\theta \in (0,1)$. By the approximation principle, Lemma~\ref{lem:approx}, we have that both $\mu,\nu$ are weak* limits of piecewise homogeneous and averaged Young measures. The partition with respect to which the approximations are piecewise constant can be chosen the same for both $\mu$ and $\nu$ (this can be seen from the proof of the averaging principle in~\cite[Section~5.3]{KristensenRindler10YM}). Thus, by the weak* closedness proved in the first step, it suffices to show the result for homogeneous, one-directional BD-Young measures.

Assume now that we have two homogeneous, one-directional BD-Young measures $\bar{\mu},\bar{\nu} \in \BDY^\sing(a \odot b,a)$ with $\lambda_\mu(\partial Q_a) = \lambda_\nu(\partial Q_a) = 0$, which we assume to be defined on a cube with one face orthogonal to $a$. Indeed, by Remark~\ref{rem:homYM_domain} we can always assume that the averaged Young measures $\bar{\mu}, \bar{\nu}$ are defined on $Q_a$ (one can also argue by inspecting the proof of Lemma~\ref{lem:approx} to conclude that the subdivision may be chosen to consist of cubes only).

Without loss of generality we further assume $a = \ee_1$ and said cube to be the unit cube $Q = (-1/2,1/2)^d$. Thus, $[\bar{\mu}] = Eu$, $[\bar{\nu}] = Ev$ for $u(x) = \eta_1 b x_1$, $v(x) = \eta_2 b x_1$ (we can remove any additional rigid deformation). Indeed, since $[\bar{\mu}] = \eta_1 (a \odot b) \Lcal^d \restrict Q$ ($\eta_1 > 0$), the $u$ defined before satisfies $Eu = [\bar{\mu}]$. Next, let $(u_j), (v_j) \subset \BD(Q) \cap \Crm^\infty(Q;\R^d)$ be bounded sequences such that $\Ecal u_j \toY \bar{\mu}$, $\Ecal v_j \toY \bar{\nu}$ and $u_j = u$, $v_j = v$ on $\partial Q$, see Lemma~\ref{lem:good_genseq}.

Now split $Q$ into a slice $S_1 = (-1/2,\theta-1/2) \times (1/2,1/2)^{d-1}$ with volume $\theta$ and a slice $S_2 = (\theta-1/2,1/2) \times (1/2,1/2)^{d-1}$ with volume $1-\theta$. Then cover $S_1, S_2$ with disjoint re-scaled copies of $Q$ of side length at most $1/j$ (for instance arranged in strips), namely
\[
  S_1 = \left( \bigcup_{k \in \N} p_{jk} + \eps_j Q \right) \cup N_1,  \qquad
  S_2 = \left( \bigcup_{k \in \N} q_{jk} + \delta_j Q  \right) \cup N_2,
\]
where $\eps_j, \delta_j \leq 1/j$ and $\abs{N_1} = \abs{N_2} = 0$. Set
\[
  w_j := \begin{cases}
    \eps_j u_j \Bigl( \frac{x-p_{jk}}{\eps_j} \Bigr) + u(p_{jk})  &
      \text{if $x \in p_{jk} + \eps_j Q$, $k \in \N$,}\\
    \delta_j v_j \Bigl( \frac{x-q_{jk}}{\delta_j} \Bigr) + v(q_{jk}) + \beta b &
      \text{if $x \in q_{jk} + \delta_j Q$, $k \in \N$,}
  \end{cases}
\]
with $\beta \in \R$ such that there is no jump between $S_1$ and $S_2$, i.e.\ $\beta = (\eta_1-\eta_2)(\theta-1/2)$.
Then,
\[
  \Ecal w_j := \begin{cases}
    \Ecal u_j \Bigl( \frac{x-p_{jk}}{\eps_j} \Bigr)  &
      \text{if $x \in p_{jk} + \eps_j Q$, $k \in \N$,}\\
    \Ecal v_j \Bigl( \frac{x-q_{jk}}{\delta_j} \Bigr) &
      \text{if $x \in q_{jk} + \delta_j Q$, $k \in \N$.}
  \end{cases}
\]
Moreover, $(w_j)$ is uniformly bounded in $\BD(Q)$ and $\Ecal w_j \toY \gamma \in \BDY^\sing(a \odot b,a)$. Effectively, in $S_1, S_2$ we are repeating the averaging construction also underlying Lemma~\ref{lem:averaging} and so, similar conclusions to the ones in that lemma hold. In particular, we get for $\phi \otimes h \in \Ebf^\sing(Q)$ with $\phi \in \Crm(Q)$, $h \in \Crm(\Rdds)$ that
\[
  \ddprb{\phi \otimes h, \gamma}
  = \frac{\ddprb{\ONE \otimes h, \bar{\mu}}}{\abs{Q}} \int_{S_1} \phi(x) \dd x
  + \frac{\ddprb{\ONE \otimes h, \bar{\nu}}}{\abs{Q}} \int_{S_2} \phi(x) \dd x
\]
Finally, apply the averaging principle, Lemma~\ref{lem:averaging}, to $\gamma$ to get $\bar{\gamma} \in \BDY^\sing(a \odot b,a)$, which by~\eqref{eq:averaged_GYM_action} has the property
\begin{align*}
  \ddprb{\phi \otimes h, \bar{\gamma}}
  &= \ddprb{\ONE \otimes h, \gamma} \dashint_Q \phi(x) \dd x \\
  &= \Bigl[ \theta \ddprb{\ONE \otimes h, \bar{\mu}} + (1-\theta) \ddprb{\ONE \otimes h, \bar{\nu}} \Bigr] \dashint_Q \phi(x) \dd x \\
  &= \theta \ddprb{\phi \otimes h, \bar{\mu}} + (1-\theta) \ddprb{\phi \otimes h, \bar{\nu}}
\end{align*}
for $\phi \otimes h$ as above. This shows the claim for homogeneous, one-directional BD-Young measures.

\proofstep{Step~3: Convexity of $\BDY^\sing(a \odot b,a)$.}
To conclude the proof we show that the set of $\mu\in BDY^\sing(a \odot b,a)$ such that $\lambda_\mu(\partial Q_a) = 0$ is weakly* dense in $\BDY^\sing(a \odot b,a)$. Indeed, assume that $a = \ee_1$, $\sigma \in \BDY^\sing(\ee_1 \odot b, \ee_1)$ and that for $(u_j) \subset \BD(Q) \cap \Crm^\infty(Q;\R^d)$ with $u_j|_{\partial Q} = bg(x_1)|_{\partial Q}$ for some $g \in \BV(-1/2,1/2)$, we have $\Ecal u_j \toY \sigma$. We consider $g$ to be extended  continuously to all of $\R$. Then, define for $\alpha > 1$,
\[
  u_j^\alpha(x) := u_j(\alpha x_1,x_2,\ldots,x_d).
\]
It is not hard to see that $\Ecal u_j^\alpha \toY \sigma^\alpha$ for some $\sigma^\alpha \in \BDY^\sing(\ee_1 \odot b, \ee_1)$ with  $\lambda_\sigma (\partial Q_a) = 0$ and $\sigma^\alpha \toweakstar \sigma$ as $\alpha \todown 1$. This and the previous step easily implies the convexity of $\BDY^\sing(a \odot b,a)$.
\end{proof}

\section{Local characterization} \label{sc:local}

We first show the characterization result for tangent Young measures, i.e.\ those Young measures originating from Propositions~\ref{prop:localize_reg} and~\ref{prop:localize_sing} above.

\subsection{Characterization for regular blow-ups}

With the definition of $\Ebf^\reg$, $\Ybf^\reg(A_0)$, $\BDY^\reg(A_0)$ from Section~\ref{ssc:loc}, we have the following result about the characterization at regular points:

\begin{proposition} \label{prop:local_reg}
Let $\sigma \in \Ybf^\reg(A_0)$ for $A_0 \in \Rdds$ and assume that
\[
  h(A_0) \leq \dprb{h,\sigma_y} + \dprb{h^\#,\sigma^\infty_y} \frac{\di \lambda_\sigma}{\di \Lcal^d}(y)
\]
for all symmetric-quasiconvex $h \in \Crm(\Rdds)$ with linear growth at infinity. Then, $\sigma \in \BDY^\reg(A_0)$.
\end{proposition}

Our proof here is quite concise since it is very close to Kinderlehrer \& Pedregal's original argument~\cite{KinderlehrerPedregal91, KinderlehrerPedregal94} and also essentially the same as the one given for~\cite[Proposition~3.2]{Rindler14YM}.

\begin{proof}
\proofstep{Step 1.}
First, by Lemma~\ref{lem:BDYregsing_convex}, the set $\BDY^\reg(A_0)$ is weakly*-closed and convex (considered as a subset of $(\Ebf^\reg)^*$).

We will show below that for every weakly*-closed affine half-space $H$ in $(\Ebf^\reg)^*$ with $\BDY^\reg(A_0) \subset H$, we have $\sigma \in H$. Then the Hahn--Banach theorem will imply that $\sigma \in \BDY^\reg(A_0)$. Fix such a weakly* closed half-space $H$. By standard arguments from functional analysis, see for example~\cite[Theorem~V.1.3]{Conway90book}, there exists $f_H \in \Ebf^\reg$ such that
\[
  H = \setb{ e^* \in (\Ebf^\reg)^* }{ e^*(f_H) \geq \kappa }.
\]
In particular,
\[
  \ddprb{f_H,\mu} \geq \kappa  \qquad\text{for all $\mu \in \BDY^\reg(A_0)$.}
\]
We will show $\ddprn{f_H,\sigma} \geq \kappa$, whereby $\sigma \in H$.

\proofstep{Step 2.}
For the the symmetric-quasiconvex envelope $SQf_H$ of $f_H$ it holds that $SQf_H(A_0) > - \infty$. Indeed, otherwise we could find $w \in \Wrm_{A_0 x}^{1,\infty}(\Bbb^d;\R^m)$, that is $w \in \Wrm^{1,\infty}(\Bbb^d;\R^m)$ and $w(x) = A_0 x$ for all $x \in \partial \Bbb^d$, such that
\[
  \int_{\Bbb^d} f_H(\Ecal w(y)) \dd y  <  \kappa.
\]
Then, using the generalized Riemann--Lebesgue Lemma, Corollary~\ref{cor:Riemann_Lebesgue}, there exists $\mu \in \BDY^\reg(A_0)$ with
\[
  \ddprb{f_H,\mu} = \int_{\Bbb^d} f_H(\Ecal w(y)) \dd y < \kappa,
\]
which is a contradiction.

For a fixed $\eps > 0$, the function
\[
  g_\eps(A) := f_H(A) + \eps \abs{A},  \qquad A \in \R^{m \times d},
\]
lies in $\Ebf^\reg$. It holds that
\[
  SQg_\eps(A_0) \geq SQf_H(A_0) + \eps \abs{A_0} > -\infty.
\]
Consequently, the function $SQg_\eps$ is symmetric-quasiconvex, see the appendix of~\cite{KinderlehrerPedregal91}. Moreover $SQg_\epsilon(A) \leq (M+1)(1+\abs{A})$. By Lemma~2.5 in~\cite{Kristensen99}, even $\abs{SQg_\epsilon(A)} \leq \tilde{M}(1+\abs{A})$ for some $\tilde{M} = \tilde{M}(d,m,M) > 0$. Using $g_\eps \geq SQg_\eps$, $g_\eps^\infty \geq (SQg_\eps)^\#$, and the assumption,
\begin{equation} \label{eq:reg_est_2}
  \hspace{-10pt}\ddprb{g_\eps,\sigma} \geq \int_{\Bbb^d} \dprb{SQg_\eps,\sigma_y} + \dprb{(SQg_\eps)^\#,\sigma_y^\infty} \, \frac{\di \lambda_\sigma}{\di \Lcal^d}(y) \dd y \geq SQg_\eps(A_0) \abs{\Bbb^d}.
\end{equation}

Next, take a sequence $(w_j) \subset \Wrm_{A_0 x}^{1,\infty}(\Bbb^d;\R^m)$ with
\[
  \dashint_{\Bbb^d} g_\eps(\Ecal w_j(y)) \dd y  \to  SQg_\eps(A_0).
\]
Moreover, possibly discarding leading elements of the sequence $(w_j)$,
\begin{align*}
  SQg_\eps(A_0) + 1 &\geq \dashint_{\Bbb^d} g_\eps(\Ecal w_j(y)) \dd y \geq \dashint_{\Bbb^d} SQf_H(\Ecal w_j(y)) + \eps \abs{\Ecal w_j(y)} \dd y \\
  &\geq SQf_H(A_0) + \frac{\eps}{\abs{\Bbb^d}} \cdot \norm{\Ecal w_j}_{\Lrm^1}.
\end{align*}
Hence, the sequence $(w_j)$ is uniformly bounded in $\BD(\Bbb^d)$ and, up to a subsequence, $\Ecal w_j \toY \mu \in \BDY(\Bbb^d)$. Apply Lemma~\ref{lem:averaging} to replace $\mu$ by its averaged version $\bar{\mu} \in \BDY^\reg(A_0)$. Then,
\[
  \ddprb{g_\eps,\bar{\mu}} = \ddprb{g_\eps,\mu} = \lim_{j \to \infty} \int_{\Bbb^d} g_\eps(\Ecal w_j(y)) \dd y = SQg_\eps(A_0) \abs{\Bbb^d}.
\]
Combining with~\eqref{eq:reg_est_2}, we get
\begin{align*}
  \ddprb{f_H,\sigma} &= \ddprb{g_\eps,\sigma} - \eps \ddprb{\ONE \otimes \abs{\frarg},\sigma} \\
  &\geq SQg_\eps(A_0) \abs{\Bbb^d} - \eps \ddprb{\ONE \otimes \abs{\frarg},\sigma} \\
  &= \ddprb{g_\eps,\bar{\mu}} - \eps \ddprb{\ONE \otimes \abs{\frarg},\sigma} \\
  &\geq \ddprb{f_H,\bar{\mu}} - \eps \ddprb{\ONE \otimes \abs{\frarg},\sigma} \\
  &\geq \kappa - \eps \ddprb{\ONE \otimes \abs{\frarg},\sigma}
\end{align*}
since $\bar{\mu} \in \BDY^\reg(A_0) \subset H$. Now let $\eps \todown 0$ to get $\ddpr{f_H,\sigma} \geq \kappa$. Thus, $\sigma \in H$.
\end{proof}

\subsection{Characterization for singular blow-ups}

Here we will prove the singular analogue of Proposition~\ref{prop:local_reg}.

\begin{proposition} \label{prop:local_sing}
$\Ybf^\sing_0(a \odot b,\xi) = \BDY^\sing_0(a \odot b,\xi)$ for all $a,b \in \R^d \setminus \{0\}$, $\xi \in \{a,b\}$.
\end{proposition}

The preceding proposition is surprising since it says that \emph{every} singular Young measure in $\Ybf^\sing_0(a \odot b,\xi)$ is generated by a sequence of symmetric derivatives of BD-functions.

We first record the following lemma, which is a direct consequence of the main result in~\cite{KirchheimKristensen16}:

\begin{lemma} \label{lem:auto_Jensen}
Let $\mu \in \Mcal_1(X;\Rdds)$ be a probability measure with barycenter $[\mu] = \dpr{\id,\mu} = a \odot b$ for some $a,b \in \R^d$, and let $h \in \Crm(\Rdds)$ be positively $1$-homogeneous and symmetric-quasiconvex. Then, the Jensen-inequality
\[
  h(a \odot b) = h(\dpr{\id,\mu}) \leq \dpr{h,\mu}.
\]
holds.
\end{lemma}

To prove this result we recall a version of the main result of~\cite{KirchheimKristensen16}:

\begin{theorem} \label{thm:KirchheimKristensen}
Let $h^\infty \colon \Rdds \to \R$ be positively one-homogeneous and symmetric quasiconvex. Then, $h^\infty$ is convex at every matrix $a \odot b$ for $a, b \in \R^d$, that is, there exists an affine function $g \colon \Rdds \to \R$ with
\[
  h^\infty(a \odot b) = g(a \odot b)  \qquad\text{and}\qquad
  h^\infty \geq g.
\]
\end{theorem}

\begin{proof}[Proof of Lemma~\ref{lem:auto_Jensen}]
By the preceding theorem, $h$ is actually \emph{convex} at matrices $a \odot b$, that is, the Jensen inequality holds for measures with barycenter $a \odot b$, such as our $\mu$.
\end{proof}

The following lemma on \enquote{artificial concentrations} will be crucial in the sequel:

\begin{lemma} \label{lem:artconc}
Let $\nu \in \BDY(S)$, where for some $z_0 \in Q_a$, $R > 0$, $a \in \Sbb^{d-1}$,
\[
  S = S(z_0,R):= \setb{ x \in Q_a }{ \abs{(x-z_0)\cdot a} < R }.
\]
Assume further that there exists a sequence $(v_j) \subset \BD(S)$ with $Ev_j \toY \nu$ and $v_j(x) = bg(x \cdot a)$ on $\partial S$ for some $b \in \R^d$ and $g \in \BV(\R)$. Then, there exists $\hat{\nu} \in \BDY^\sing(a \odot b,a)$ such that
\begin{equation} \label{eq:artconc_eq}
  \ddprb{\ONE \otimes h,\nu} = \ddprb{\ONE \otimes h, \hat{\nu}}.
\end{equation}
for all positively $1$-homogeneous $h \in \Crm(\Rdds)$. The condition on the generating sequence is in particular satisfied if $[\nu] = Eu$ for some $u \in \BD(S)$ with the property that $u(x) = b g(x \cdot a)$ on $\partial S$.
\end{lemma}

\begin{proof}
Note that by Lemma~\ref{lem:good_genseq} if $[\nu] = Eu$ for some $u \in \BD(S)$ with the property that $u(x) = b g(x \cdot a)$ on $\partial S$ there always exists $(u_j) \subset \BD(S) \cap \Crm^\infty(S;\R^d)$ with $\Ecal u_j \toY \nu$ and $u_j(x) = bg(x \cdot a)$ on $\partial S$. Hence the second part of the statement follows from the first.

Up to a translation and a one-dimensional scaling, we can assume  that $x_0=0$ and $R=1$. Let
\[
  S(z_0,r) := \setb{ x \in Q_a }{ \abs{(x-z_0)\cdot a} < r },  \qquad
  z_0 \in Q_a, \; r > 0,
\]
and define $w_j\in BD(S)$ by
\[
  w_j(x) := \begin{cases}
    u_j(jx)  &\text{if $x \in S(0,1/j)$,}\\
    bg(-1)  &\text{if $x \cdot a \leq -1/j$,}\\
    bg(+1)  &\text{if $x \cdot a \geq 1/j$,}
  \end{cases}
  \qquad
  x \in Q_a,
\]
where $g(\pm 1)$ is defined in the sense of trace. It can be seen that $\Ecal w_j$ generates a Young measure $\mu \in \BDY(Q_a)$ with $\mu_x = \delta_0$ almost everywhere since $\Ecal w_j \to 0$ in measure. Furthermore, for all positively $1$-homogeneous $h \in \Crm(\Rdds)$, we get from the $a$-directionality of $g$ that
\begin{align*}
  \ddprb{\ONE \otimes h, \mu} &= \lim_{j\to\infty} \int_{S(0,1/j)} jh(\Ecal u_j(jx)) \dd x\\
  &= \lim_{j\to\infty} \int_{Q_a} h(\Ecal u_j) \dd x\\
  &= \ddprb{\ONE \otimes h, \nu}.
\end{align*}
Now apply the averaging principle, Lemma~\ref{lem:averaging}, to $\mu$ to get $\hat{\nu} \in \BDY^\sing(a \odot b,a)$ with the property~\eqref{eq:artconc_eq}. Indeed, the maps $y \mapsto \hat{\nu}_y^\infty$, $x \mapsto \hat{\nu}_y^\infty$ are constant (a.e.) and $\lambda_{\hat{\nu}}$ is a multiple of Lebesgue measure, so it only remains to check that $[\hat{\nu}] = (a \odot b) \lambda_{\hat{\nu}}$. To see the latter result, it suffices to observe using an integration by parts that ($n$ being the outward unit normal)
\begin{align*}
  [\hat{\nu}] &= \ddprb{\ONE \otimes \id, \mu} \\
  &= \lim_{j\to\infty} \int_{Q_a} \Ecal w_j \dd x \\
  &= \lim_{j\to\infty} \int_{\partial Q_a} w_j \odot n \dd x \\
  &= b \bigl( g(+1)-g(-1) \bigr) \odot a.
\end{align*}
Since also $\lambda_{\hat{\nu}}$ is a multiple of Lebesgue measure, we conclude $[\hat{\nu}] = (a \odot b) \lambda_{\hat{\nu}}$.
\end{proof}

\begin{remark}\label{rem:BDpolar_necessary}
The preceding result is in fact the reason why we need the singular structure theorem in BD, Theorem~\ref{thm:BDpolar}, as opposed to a mere rigidity result as in~\cite{Rindler11}. Indeed, the above Lemma will  play a key role in the proof of Proposition~\ref{prop:local_sing} and in order to apply it  one is forced to require in the definition of $\Ybf^\sing(a\odot b,a)$ the $a$-directionality of $\lambda_\sigma$. Lemma~\ref{lem:very_good_blowups}, which relies on Theorem~\ref{thm:BDpolar}, then  implies that it is not restrictive to consider tangent Young measure lying in  $\Ybf^\sing(a\odot b,a)$, see  Step~2 in the proof of Proposition~\ref{prop:final}.
\end{remark}

We can now turn to the main aim of this section:

\begin{proof}[Proof of Proposition~\ref{prop:local_sing}]
Without loss of generality we assume that $\xi = a$.

\proofstep{Step~1.}
We only need to show that for $\sigma \in \Ybf^\sing_0(a \odot b,a)$ we also have $\sigma \in \BDY^\sing_0(a \odot b,a)$ for all $a,b \in \R^d \setminus \{0\}$. Like in Proposition~\ref{prop:local_reg}, we will employ the Hahn--Banach theorem to show that for any weakly*-closed affine half-space $H \subset \Ebf^\sing(a)^*$ with $\BDY^\sing(a \odot b,a) \subset H$ it holds that $\sigma \in H$. Then, since $\BDY^\sing(a \odot b,a)$ is weakly* closed and convex by Lemma~\ref{lem:BDYregsing_convex}, it will follow that $\sigma \in \BDY^\sing(a \odot b,a)$. There exists $f_H \in \Ebf^\sing(a)$ and $\kappa \in \R$ such that
\[
  H = \setb{ G \in \Ebf^\sing(a)^* }{ G(f_H) \geq \kappa }.
\]
Since we assumed $\BDY^\sing(a \odot b,a) \subset H$, we have in particular
\[
  \ddpr{f_H, \nu} \geq \kappa  \qquad
  \text{for all $\nu \in \BDY^\sing(a \odot b,a)$.}
\]
We need to show that $\ddpr{f_H, \sigma} \geq \kappa$ in order to conclude that $\sigma \in H$.

\proofstep{Step~2.}
Fix $\eps, \delta > 0$. We define
\[
  f_H^\eps(x,A) := f_H(x,A) + \eps \abs{A},  \qquad
  x \in Q_a,\; A \in \Rdds,
\]
which lies in $\Ebf^\sing(a)$. Next, subdivide $Q_a$ into slices $S_1, \ldots, S_n$ along the $a$-axis ($a$ orthogonal to the \enquote{long} face of the slices), that is, the $S_k$ are of the form
\[
  S_k = S(z_k,r_k) := \setb{ x \in Q_a }{ \abs{(x-z_k)\cdot a} < r_k }
\]
for some $z_k \in Q_a$, $r_k > 0$. Assume furthermore that the $S_k$ are chosen in such a way that
\begin{equation} \label{eq:RfH_cont}
  \absb{Rf_H^\eps(x,A) - Rf_H^\eps(y,A)} \leq \delta  \qquad
  \text{for all $x,y \in S_k$, $A \in \cl{\Bdds}$,}
\end{equation}
since by assumption $Rf_H^\eps$ is uniformly continuous and one-directional, where $R$ is defined~\eqref{eq:R}. Moreover, we can require $\lambda_\sigma(\partial S_k) = 0$ for all $k=1,\ldots,n$. We will show that in each $S_k$ there exists a point $z_k$ at which the following properties hold for the symmetric-quasiconvex hull $SQf_H^\eps(z_k,\frarg)$ of $f_H^\eps(z_k,\frarg)$:
\begin{enumerate}[(A)]
  \item $SQf_H^\eps(z_k,\frarg)$ is finite, symmetric-quasiconvex, and positively $1$-homogeneous.
  \item There exists a \enquote{recovery sequence} $(\psi^{(k)}_j)_j \subset \Wrm^{1,\infty}_{b(x\cdot a)}(S_k;\Rdds)$ such that
\begin{equation} \label{eq:psi_SQfH} \qquad
  \dashint_{S_k} f_H^\eps(z_k, \Ecal\psi^{(k)}_j(x)) \dd x \to SQf_H^\eps(z_k,a \odot b)
\end{equation}
and such that for a constant $C_\eps$, which is independent of $\delta$, it holds that
\begin{equation} \label{eq:psi_bound} \qquad
  \supmod_j \norm{\Ecal \psi^{(k)}_j}_{\Lrm^1} \leq C_\eps \abs{S_k}.
\end{equation}
\end{enumerate}

For the finiteness in~(A), by standard arguments for (symmetric-)quasiconvex functions, see the appendix of~\cite{KinderlehrerPedregal91}, we need only show $SQf_H(z_k,a \odot b) > -\infty$ (at one point only). To see $SQf_H(z_k,a \odot b) > -\infty$, assume to the contrary that there exists an open slice
\[
  S(z_0,r) := \setb{ x \in Q_a }{ \abs{(x-z_0)\cdot a} < r },  \qquad
  z_0 \in Q_a, \; r > 0,
\]
with the property that $SQf_H(z,a \odot b) = -\infty$ for all $z \in S(z_0,r)$ (recall that $f_H$, hence also $SQf_H$, by definition is $a$-directional). By definition then we can find $\psi_z \in \Wrm^{1,\infty}_{b(x\cdot a)}(Q_a;\R^d)$ with
\[
  \dashint_{Q_a} f_H(z,\Ecal \psi_z(y)) \dd y < \frac{\kappa}{\abs{S(z_0,r)}} - 1.
\]
Indeed, there exists $\tilde{\psi}_z \in \Wrm^{1,\infty}_0(Q_a;\R^s)$ such that
\[
  \dashint_{Q_a} f_H(z,a \odot b + \Ecal \tilde{\psi}_z(y)) \dd y < \frac{\kappa}{\abs{S(z_0,r)}} - 1.
\]
Then, the assertion follows with $\psi_z(x) := \tilde{\psi}_z(x) + b(x\cdot a)$.

Furthermore, we can assume that the map $z \mapsto \psi_z$ depends only on $z \cdot a$ (by the $a$-directionality of $f_H$), and that by the uniform continuity of $Rf_H$ at each $z \in S(z_0,r)$ there exists $\eta(z) > 0$ such that
\[
  \abs{f_H(x,A) - f_H(z,A)} \leq \frac{1 + \abs{A}}{1 + \norm{\Ecal \psi_z}_{\Lrm^1}} \qquad
  \text{for all $x \in S(z,\eta(z))$, $A \in \Rdds$.}
\]
Now use the Vitali covering theorem (in $\R$) to cover $\Lcal^d$-almost all of $S(z_0,r)$ with slices $S_i = S(z_i,r_i)$ such that $r_i < \eta(z_i)$ ($i \in \N$). The generalized Riemann--Lebesgue lemma, Corollary~\ref{cor:Riemann_Lebesgue}, then allows us to find $\mu_i \in \BDY(S_i)$ with underlying deformation $b(x \cdot a)$ and
\[
  \ddprb{f_H,\mu_i} = \int_{S_i} \dashint_{Q_a} f_H(x,\Ecal \psi_{z_i}(y)) \dd y \dd x.
\]
Thus, glueing these $\mu_i$ together and applying Lemma~\ref{lem:artconc} separately in each $S_i$, we get $\mu \in \BDY^\sing(a \odot b,a)$ such that
\begin{align*}
  \ddprb{f_H, \mu} &= \sum_{i=1}^\infty \ddprb{f_H|_{S_i}, \mu_i} \\
  &= \sum_{i=1}^\infty \int_{S_i} \dashint_{Q_a} f_H(x,\Ecal \psi_{z_i}(y)) \dd y \dd x \\
  &\leq \sum_{i=1}^\infty \biggl( \dashint_{Q_a} f_H(z_i,\Ecal \psi_{z_i}(y)) \dd y + 1 \biggr) \abs{S_i} \\
  &< \kappa,
\end{align*}
in contradiction to $\ddpr{f_H,\mu} \geq \kappa$ since $\mu \in H$.

The symmetric-quasiconvexity and the positive $1$-homogeneity of $SQf_H^\eps(z_k,\frarg)$ are now easy to see by standard techniques, see for instance the appendix to~\cite{KinderlehrerPedregal91}, which concerns quasiconvexity, but the methods adapt.

To show~(B), we first recall that from the fact that for separately convex functions an upper $p$-growth bound also implies a lower $p$-growth bound (with a different constant), where $1 \leq p < \infty$, see Lemma~2.5 in~\cite{Kristensen99}. Thus, it follows that
\begin{equation} \label{eq:SQfH_bound}
  \abs{SQf_H^\eps(z_k,A)} \leq \tilde{M}(1+\abs{A})
\end{equation}
for some $\tilde{M} = \tilde{M}(d,M)$, which is independent of $z_k$. Let $(\psi_j^{(k)})_j \subset \Wrm^{1,\infty}_{b(x\cdot a)}(S_k;\R^d)$ be a minimizing sequence for
\[
  \psi \mapsto \dashint_{S_k} f_H^\eps(z_k,\Ecal \psi(x)) \dd x, \qquad
  \psi \in \Wrm^{1,\infty}_{b(x\cdot a)}(S_k;\R^d).
\]
By definition, this sequence $(\psi_j^{(k)})_j$ satisfies~\eqref{eq:psi_SQfH}. Further, we may estimate, using the symmetric-quasiconvexity and $SQf_H(z_k,A) + \eps\abs{A} \leq SQf_H^\eps(z_k,A) \leq f_H^\eps(z_k,A)$ that 
\begin{align*}
  &SQf_H(z_k,a \odot b) + \eps \dashint_{S_k} \abs{\Ecal \psi_j^{(k)}(x)} \dd x \\
  &\qquad\leq \dashint_{S_k} SQf_H(z_k,\Ecal \psi_j^{(k)}(x)) + \eps\abs{\Ecal \psi_j^{(k)}(x)} \dd x\\
  &\qquad\leq \dashint_{S_k} f_H^\eps(z_k,\Ecal \psi_j^{(k)}(x)) \dd x\\
  &\qquad\leq SQf_H^\eps(z_k,a \odot b) + 1
\end{align*}
where we have also discarded some leading elements from the sequence $(\psi_j^{(k)})_j$. Thus, by~\eqref{eq:SQfH_bound},
\[
  \int_{S_k} \abs{\Ecal \psi_j^{(k)}(x)} \dd x
  \leq \frac{2 \tilde{M} (1 + \abs{a \odot b}) + 1}{\eps} \abs{S_k}
  =: C_\eps \abs{S_k},
\]
which is~\eqref{eq:psi_bound}.

\proofstep{Step~3.}
Now, with the $z_k$'s chosen in each $S_k$ to satisfy~(A),~(B), we estimate as follows, using~\eqref{eq:RfH_cont} and the Jensen-type inequality from Lemma~\ref{lem:auto_Jensen},
\begin{align}
  \ddprb{f_H,\sigma} &= \ddprb{f_H^\eps,\sigma} - \eps \ddprb{\ONE \otimes \abs{\frarg},\sigma}  \notag\\
  &= \sum_{k=1}^n \int_{S_k} \dprb{f_H^\eps(z,\frarg),\sigma^\infty_z} \dd \lambda_\sigma(z) - \eps \ddprb{\ONE \otimes \abs{\frarg},\sigma}  \notag\\
  &\geq \sum_{k=1}^n \int_{S_k} \dprb{f_H^\eps(z_k,\frarg),\sigma^\infty_z} \dd \lambda_\sigma(z) - (\eps+\delta) \ddprb{\ONE \otimes (1+\abs{\frarg}),\sigma}  \notag\\
  &\geq \sum_{k=1}^n \int_{S_k} \dprb{SQf_H^\eps(z_k,\frarg),\sigma^\infty_z} \dd \lambda_\sigma(z) - (\eps+\delta) \ddprb{\ONE \otimes (1+\abs{\frarg}),\sigma}  \notag\\
  &\geq \sum_{k=1}^n SQf_H^\eps(z_k,a \odot b) \, \lambda_\sigma(S_k) - (\eps+\delta) \ddprb{\ONE \otimes (1+\abs{\frarg}),\sigma}  \label{eq:fH_est}
\end{align}

By~(B), for every $k=1,\ldots,n$ there exists a \enquote{recovery sequence} $(\psi^{(k)}_j)_j \subset \Wrm^{1,\infty}_{b(x\cdot a)}(S_k;\Rdds)$ with $\supmod_j \norm{\Ecal \psi^{(k)}_j}_{\Lrm^1} \leq C_\eps \abs{S_k}$ and
\[
  \dashint_{S_k} f_H^\eps(z_k, \Ecal\psi^{(k)}_j(x)) \dd x \to SQf_H^\eps(z_k,a \odot b).
\]
Now define
\[
  w_j(x) := \psi^{(k)}_j(x) \frac{\lambda_\sigma(S_k)}{\abs{S_k}} + \beta^k b \quad
  \text{if $x \in S_k$,}
\]
where the $\beta^k \in \R$ are chosen to remove any jumps in the definition of $w_j$ (recall that $\lambda_\sigma$ is one-directional and that $\psi^{(k)}_j(x) = b(x \cdot a)$ for $x \in \partial S_k$). It is not difficult to see, using the norm bound~\eqref{eq:psi_bound} from~(B) and Poincar\'{e}'s inequality~\eqref{eq:Poincare_ext}, that after adding some rigid deformations (suppressed in the following) the sequence $(w_j)_j$ is uniformly bounded in $\BD(Q_a)$ and that
\[
  \int_{S_k} f_H^\eps(z_k,\Ecal w_j(x)) \dd x \to SQf_H^\eps(z_k,a \odot b) \lambda_\sigma(S_k)  \qquad
  \text{for $k = 1,\ldots,n$.}
\]
Combining with~\eqref{eq:fH_est}, so far we have shown that
\[
  \ddprb{f_H,\sigma} \geq \lim_{j\to\infty} \sum_{k=1}^n \int_{S_k} f_H^\eps(z_k,\Ecal w_j(x)) \dd x - (\eps+\delta) \ddprb{\ONE \otimes (1+\abs{\frarg}),\sigma}.
\]
Let $\mu^{(\delta)} \in \BDY(Q_a)$ be the BD-Young measure generated by $(\Ecal w_j)_j$ (up to a non-relabeled subsequence). Note that for $\mu^{(\delta)} \restrict S_k$ we have a generating sequence with boundary values $\alpha_k b(x \cdot a) + \beta_k b$ on $\partial S_k$, where $\alpha_k = \lambda_\sigma(S_k) / \abs{S_k}$.
Apply Lemma~\ref{lem:artconc}, separately in each $S_k$, to replace $\mu^{(\delta)}$ by $\hat{\mu}^{(\delta)} \in \BDY^\sing(a \odot b,a)$ with
\[
  \ddprBB{\,\sum_{k=1}^n \ONE_{S_k} \otimes f_H^\eps(z_k,\frarg), \hat{\mu}^{(\delta)}}
  = \lim_{j\to\infty} \sum_{k=1}^n \int_{S_k} f_H^\eps(z_k,\Ecal w_j(x)) \dd x.
\]
Thus,
\begin{align*}
  \ddprb{f_H,\sigma}
  &\geq \ddprBB{\,\sum_{k=1}^n \ONE_{S_k} \otimes f_H^\eps(z_k,\frarg), \hat{\mu}^{(\delta)}} - (\eps+\delta) \ddprb{\ONE \otimes (1+\abs{\frarg}),\sigma} \\
  &\geq \ddprb{f_H^\eps,\hat{\mu}^{(\delta)}} - \delta \ddprb{\ONE \otimes (1+\abs{\frarg}),\hat{\mu}^{(\delta)}} - (\eps+\delta) \ddprb{\ONE \otimes (1+\abs{\frarg}),\sigma} \\
  &\geq \kappa - \delta \ddprb{\ONE \otimes (1+\abs{\frarg}),\hat{\mu}^{(\delta)}} - (\eps+\delta) \ddprb{\ONE \otimes (1+\abs{\frarg}),\sigma}.
\end{align*}
Here, for the last line we used $\ddprb{f_H^\eps,\hat{\mu}^{(\delta)}} \geq \ddprb{f_H,\hat{\mu}^{(\delta)}} \geq \kappa$ since $\hat{\mu}^{(\delta)} \in H$ (recall that $H$ was such that $\BDY^\sing(a \odot b,a) \subset H$). Now, first let $\delta \to 0$, using that the $\hat{\mu}^{(\delta)}$'s are uniformly in the Young measure-sense bounded (since the bound in~\eqref{eq:psi_bound} is independent of $\delta$), and then let $\eps \todown 0$ to arrive at
\[
  \ddprb{f_H,\sigma} \geq \kappa.
\]
Hence, also $\sigma \in H$ and the Hahn--Banach argument described at the beginning of the proof yields the conclusion.
\end{proof}

\subsection{Characterization for zero-barycenter singular blow-ups}

Here, we use the following spaces (where $Q = (-1/2,1/2)^d$ is the unit cube):
\begin{align*}
  \Ebf^\sing(0) &:= \Ebf^\sing(Q;\Rdds), \\
  \Ybf^\sing(0) &:= \setb{ \sigma \in \Ybf^\sing(Q;\Rdds) }{ [\sigma] = 0 }, \\
  \BDY^\sing(0) &:= \Ybf^\sing(0) \cap \BDY^\sing(Q).
\end{align*}
Notice that for $\sigma \in \Ybf^\sing(0)$ we do not require one-directionality of $y \mapsto \sigma_y^\infty$ and $\lambda_\sigma$. We also define the spaces $\Ybf^\sing_0(0;\Rdds)$, $\BDY^\sing_0(0)$ with the additional constraint $\lambda_\sigma(\partial \Omega) = 0$.

\begin{proposition} \label{prop:local_zero}
$\Ybf^\sing_0(0) = \BDY^\sing_0(0)$.
\end{proposition}

The proof of this fact proceeds essentially in the same way as the proof for Proposition~\ref{prop:local_sing} in the previous section with some straightforward modifications:

\begin{enumerate}[(i)]
\item Wherever a direction $a$ or $\xi$ is needed, we use $a,\xi = \ee_1$.
\item The proof of the analogue of Lemma~\ref{lem:BDYregsing_convex} is exactly the same.
\item In the proof of the Proposition~\ref{prop:local_sing}, we can no longer assume that $f_H$ is one-directional. Thus, we need to replace the slices $S_i$ partitioning $Q_a = Q$ with rescaled cubes $Q_i$ covering $Q$, for instance in a regular lattice; the same holds for the slices $S(z_0,r)$. By averaging via Lemma~\ref{lem:averaging}, we may also conclude that we can get $\mu, \hat{\mu}^{(\delta)} \in \BDY^\sing(0)$.
\end{enumerate}

\section{Proof of Theorem~\ref{thm:BDY_charact}} \label{sc:proof}

First, we remark that the necessity part of our theorem is precisely the assertion of Theorem~4 of~\cite{Rindler11}. It remains to show the \enquote{sufficiency} part:

\begin{proposition}\label{prop:final}
Let $\nu \in \Ybf(\Omega;\Rdds)$ with $[\nu] = Eu$ for some $u \in \BD(\Omega)$. If for all symmetric-quasiconvex $h \in \Crm(\R_\sym^{d \times d})$ with linear growth at infinity the Jensen-type inequality
\begin{equation} \label{eq:Jensen}
  h \biggl( \dprb{\id,\nu_x} + \dprb{\id,\nu_x^\infty} \frac{\di \lambda_\nu}{\di \Lcal^d}(x) \biggr)
    \leq \dprb{h,\nu_x} + \dprb{h^\#,\nu_x^\infty} \frac{\di \lambda_\nu}{\di \Lcal^d}(x)
\end{equation}
holds for $\Lcal^d$-almost every $x \in \Omega$, then $\nu \in \BDY(\Omega)$.
\end{proposition}

\begin{proof}
Note that we may additionally assume $\lambda_\nu(\partial \Omega) = 0$ by embedding the problem into a larger domain and extending all involved maps by zero to this larger domain. This introduces an additional singular part, but this does not impinge the validity of~\eqref{eq:Jensen} on the larger domain and nothing needs to be assumed on the singular part.

We argue by considering regular and singular points separately.

\proofstep{Step~1.}
Let $x_0 \in \Omega$ be a regular point, i.e.\ a point where the regular localization principle in Proposition~\ref{prop:localize_reg} holds; this is the case for $\Lcal^d$-almost every point of $\Omega$. From said result we get the existence of a regular tangent Young measure $\sigma \in \Ybf^\reg(P_0)$, where
\[
  P_0 = \dprb{\id, \nu_{x_0}} + \dprb{\id, \nu_{x_0}^\infty} \frac{\di \lambda_\nu}{\di \Lcal^d}(x_0).
\]
We claim that $\sigma$ satisfies the Jensen-type inequality assumed in Proposition~\ref{prop:local_reg}. Indeed,
\begin{align*}
  h(P_0) &= h \biggl( \dprb{\id, \nu_{x_0}} + \dprb{\id, \nu_{x_0}^\infty} \frac{\di \lambda_\nu}{\di \Lcal^d}(x_0) \biggr) \\
  &\leq \dprb{h,\nu_{x_0}} + \dprb{h^\#,\nu^\infty_{x_0}} \frac{\di \lambda_\nu}{\di \Lcal^d}(x_0) \\
  &= \dprb{h,\sigma_y} + \dprb{h^\#,\sigma^\infty_y} \frac{\di \lambda_\sigma}{\di \Lcal^d}(y)
\end{align*}
at $\Lcal^d$-almost every $y$. Here we used~\eqref{eq:Jensen} and the properties of regular blow-ups listed in Proposition~\ref{prop:localize_reg}. Thus, Proposition~\ref{prop:local_reg} yields that
\[
  \sigma \in \BDY^\reg(Q).
\]

\proofstep{Step~2.}
By Lemma~\ref{lem:very_good_blowups} at   $\lambda_\nu^s$-almost every $x_0 \in \Omega$, there exists a singular tangent Young measure $\sigma \in \Ybf^\sing(\xi \odot \eta, \xi)$ for some $\xi \in \Sbb^{d-1}$, $\eta \in \R^d$ and with the properties listed in Proposition~\ref{prop:localize_sing} and such that 
\[
  [\sigma] = [\sigma] \restrict Q_\xi =Ev,  \qquad \text{for some $v \in \BD(Q_\xi)$.}
\]
 In particular, again by Lemma~\ref{lem:very_good_blowups},   there exists $v \in \BD(Q_\xi)$ with $[\sigma] = Ev$ such that 
\[
  v(x) = v_0 + g(x \cdot \xi)\eta + \beta (\xi \otimes \eta)x + Rx,  \qquad \text{$x \in Q_\xi$ a.e.,}
\]
for some $v_0 \in \R^d$, $\beta \in \R$, a function $g \in \BV(\R)$, and a matrix $R \in \R^{d \times d}_\skw$. Furthermore, we have that (by properties of blow-ups, see~Theorem~2.44 in~\cite{AmbrosioFuscoPallara00})
\[
  Ev = P_0 \abs{Ev} = \dprb{\id,\sigma^\infty_y} \lambda_\sigma(\di y) \qquad\text{for}\qquad
  P_0 := \frac{\di E^s u}{\di \abs{E^s u}}(x_0).
\]
In particular, $\dpr{\id,\sigma^\infty_y} = P_0$ for $\lambda_\sigma$-almost every $y \in \R^d$. Note that if $P_0 \neq 0$, then $\lambda_\sigma$ is one-directional since $Ev = [Dg(x \cdot \xi) + \beta](\eta \odot \xi)$ is.

Now, depending on whether $P_0 = 0$ or $P_0 = a \odot b$ (these are the only two possibilities by Lemma~\ref{lem:very_good_blowups}), our $\sigma$ lies either in the space $\Ybf^\sing(P_0,\xi)$ for $\xi \in \{a,b\}$ or in the space $\Ybf^\sing(0)$. Also, we may assume that $\lambda_\sigma(\partial Q_\xi) = 0$ by a simple rescaling argument (similar to the one described in Remark~\ref{rem:good_blowup_remark}). Consequently, by either Proposition~\ref{prop:local_sing} or Proposition~\ref{prop:local_zero}, we infer
\[
  \sigma \in \BDY^\sing(Q_\xi).
\]
Our proposition, and thus Theorem~\ref{thm:BDY_charact}, is now implied by the following lemma.
\end{proof}

\begin{lemma}[Glueing] \label{lem:glueing}
If $\nu \in \Ybf(\Omega;\Rdds)$ has the property that for $(\Lcal^d + \lambda_\nu^s)$-almost every $x \in \Omega$ there exists a (regular or singular) tangent Young measure $\sigma_x \in \BDY(Q_{\xi(x)})$ to $\nu$ at $x$ for some $\xi(x) \in \Sbb^{d-1}$. Then, $\nu \in \BDY(\Omega)$.
\end{lemma}

\begin{proof}
\proofstep{Step~1.}
We know that for every tangent Young measure $\sigma$, there exists a sequence of radii $r_n \todown 0$ and a sequence of constants $c_n > 0$ such that $\sigma^{(n)} \toweakstar \sigma$ for
\[
  \ddprb{\phi \otimes h, \sigma^{(n)}} = c_n \ddprBB{\phi \biggl( \frac{\frarg - x_0}{r_n} \biggr) \otimes h, \nu},
\]
where
\[
  c_n = \begin{cases}
    r_n^{-d}  &\text{if $x_0$ is regular,}\\
    \bigl( \ddpr{\ONE_{Q(x_0,r_n)} \otimes \abs{\frarg}, \nu} \bigr)^{-1}  &\text{if $x_0$ is singular.}
  \end{cases}
\]
Here, $Q(x_0,r_n) = x_0 + r_n Q$ and $Q$ generically denotes the unit cube with one face normal to $a$ or $b$ if $x_0$ is a singular point with $a \odot b \neq 0$ (see Lemma~\ref{lem:very_good_blowups}) or the standard unit cube if $x_0$ is a regular point or a singular point with $a \odot b = 0$. We require also that $\lambda_\nu(\partial Q(x_0,r_n)) = \lambda_{\sigma}(\partial Q) = 0$.

We further define
\[
  u^{(n)}(y) := r_n^{d-1} c_n \bigl[u(x_0+r_n y) - [u]_{Q(x_0,r_n)} \bigr],  \qquad 
  y \in Q,
\]
where $[u]_{Q(x_0,r_n)} := \dashint_{Q(x_0,r_n)} u \dd x$. It holds that
\[
  Eu^{(n)} = c_n T^{x_0,r_n}_\# Eu.
\]
where $T^{x_0,r_n}(x) := (x-x_0)/r_n$ and $T^{x_0,r_n}_\# Eu := Eu \circ (T^{x_0,r_n})^{-1} = Eu(x_0 + r_n \frarg)$ is the push-forward of $Eu$ under $T^{x_0,r_n}$. Moreover, we can assume that by properties of blow-ups, see Lemma~3.1 of~\cite{Rindler12YM}, there is $v \in \BD(Q)$ with $[\sigma] = Ev$ and such that
\[
  u^{(n)} \to v  \quad\text{strictly}, \qquad\text{i.e.}\qquad
  u^{(n)} \toweakstar v  \;\;\text{and}\;\;  \abs{Eu^{(n)}}(Q) \to \abs{Ev}(Q),
\]

Next, take a generating sequence $(v_j) \subset \BD(Q) \cap \Crm^\infty(Q;\R^d)$ of $\sigma$ with $v_j|_{\partial Q} = v|_{\partial Q}$ and define
\[
  v_j^{(n)}(x) := \frac{1}{r_n^{d-1} c_n} v_j \biggl( \frac{x-x_0}{r_n} \biggr) + [u]_{Q(x_0,r_n)},  \qquad
  x \in Q(x_0,r_n).
\]
The trace operator in BD is strictly continuous, see Proposition~3.4 in~\cite{Babadjian15}, and $v_j|_{\partial Q} = v|_{\partial Q}$. Hence,
\begin{align*}
  &\int_{\partial Q(x_0,r_n)} \absb{v_j^{(n)} - u} \dd \Hcal^{d-1} \\
  &\qquad = r_n^{d-1} \int_{\partial Q} \absb{r_n^{1-d} c_n^{-1} v_j(y) - u(x_0 + r_n y) + [u]_{Q(x_0,r_n)}} \dd \Hcal^{d-1}(y) \\
  &\qquad = \frac{1}{c_n} \int_{\partial Q} \absb{v - u^{(n)}} \dd \Hcal^{d-1}(y).
\end{align*}
Consequently, since the boundary integral tends to zero as $n \to \infty$, for every $k \in \N$ we may select $N(x_0,k) \in \N$ so large that
\begin{equation} \label{eq:boundary_est}
  \int_{\partial Q(x_0,r_n)} \absb{v_j^{(n)} - u} \dd \Hcal^{d-1}
  \leq \frac{1}{c_n k}
  \qquad\text{for all $n \geq N(x_0,k)$ and all $j$.}
\end{equation}

\proofstep{Step~2.}
Let the set $R \subset \Omega$ contain all regular points in $\Omega$ and let $S \subset \Omega$ contain all singular points. We have $(\Lcal^d + \lambda_\nu)(\Omega \setminus (R \cup S)) = 0$, where we have also assumed that $R,S$ are Borel sets. 

Now, let $\{\phi_\ell \otimes h_\ell\} \subset \Ebf(\Omega;\Rdds)$ be a family of integrands that determine the Young measure convergence as in Lemma~\ref{lem:dense}. It follows from the proof of the regular localization principle, Proposition~\ref{prop:localize_reg}, that every regular $x_0 \in R$ is a Lebesgue point for
\[
  x \mapsto \phi_\ell(x) \biggl( \dprb{h_\ell, \nu_x} + \dprb{h_\ell^\infty, \nu^\infty_x} \frac{\di \lambda_\nu}{\di \Lcal^d}(x) \biggr),
\]
so we may choose $N(x_0,k)$ so large that for all $\ell \leq k$ and $n \geq N(x_0,k)$ it holds that
\begin{equation} \label{eq:glue_Lebesgue1}
  \dashint_{Q(x_0,r_n)} \absBB{ \dprb{h_\ell, \nu_x} + \dprb{h_\ell^\infty, \nu^\infty_x} \frac{\di \lambda_\nu}{\di \Lcal^d}(x) - \dprb{h_\ell, \nu_{x_0}} + \dprb{h_\ell^\infty, \nu^\infty_{x_0}} \frac{\di \lambda_\nu}{\di \Lcal^d}({x_0}) } \dd y \leq \frac{1}{k}.
\end{equation}
Moreover, at every singular $x_0 \in S$ we similarly choose $N(x_0,k)$ large enough so that for all $\ell \leq k$ and $n \geq N(x_0,k)$ we have
\begin{equation} \label{eq:glue_Lebesgue2}
  \dashint_{Q(x_0,r_n)} \absb{ \dprb{h_\ell^\infty, \nu^\infty_x} - \dprb{h_\ell^\infty, \nu^\infty_{x_0}} } \dd \lambda_\nu^s(x) \leq \frac{1}{k}
\end{equation}
since in the proof of the singular localization principle, Proposition~\ref{prop:localize_sing}, it is shown that every singular $x_0 \in S$ is a $\lambda_\nu^s$-Lebesgue point of
\[
  x \mapsto \dprb{h_\ell^\infty, \nu^\infty_x}.
\]

By the Morse covering theorem, see Theorem~5.51 in~\cite{AmbrosioFuscoPallara00}, we can now cover $(\Lcal^d + \lambda_\nu)$-almost all of $\Omega$ with disjoint (rotated) cubes $Q(x_0,r_n)$ as constructed above, where $n \geq N(x_0,k)$ and $r_n \leq 1/k$. Let
\[
  \Omega = \left( \bigcup_{i=1}^\infty Q_i(x_i,r_i) \right) \cup N,  \qquad (\Lcal^d + \lambda_\nu)(N) = 0,
\]
be this cover. We also denote the constructed tangent Young measure at $x_i$ (restricted to $Q_i$) by $\sigma_i \in \BDY(Q_i)$, where $Q_i$ is a unit cube. We can always require in addition that $\lambda_\nu(\partial Q_i) = \lambda_{\sigma_i}(\partial Q_i) = 0$ (the first condition holds for all but countably many radii around every point and the second was already assumed above after a rescaling argument) and
\begin{equation} \label{eq:lambda_close}
  \absb{\lambda_{\sigma_i}(Q_i) - (c_i T^{x_i,r_i}_\# \lambda_\nu^s)(Q_i)} \leq \frac{1}{k},
\end{equation}
where $c_i$ is the rescaling constant corresponding to $r_i$ and $\lambda_\nu^s$ is the singular part of $\lambda_\nu$ with respect to Lebesgue measure. 

Denote furthermore a generating sequence of $\sigma_i$ by $(v^{(i)}_j) \subset \BD(Q_i) \cap \Crm^\infty(\Omega;\R^d)$, for which we additionally require that $v^{(i)}_j|_{\partial Q_i} = v^{(i)}|_{\partial Q_i}$ with $v^{(i)}$ the underlying deformation of $\sigma_i$ as in Step~1.

Now, let $\{\phi_\ell \otimes h_\ell\} \subset \Ebf(\Omega;\Rdds)$ be a family of integrands that determine the Young measure convergence as in Lemma~\ref{lem:dense}. 
Take an index $j(i,k)$ such that
\begin{equation} \label{eq:hl}
  \absBB{\int_{Q_i} h_\ell \bigl( \Ecal v^{(i)}_{j(i,k)}(y) \bigr) \dd y
  - \ddprb{\ONE_{Q_i} \otimes h_\ell, \sigma_i}} \leq \frac{1}{k}  \qquad
  \text{for all $\ell \leq k$.}
\end{equation}

In case that $x_i$ is singular, the above estimate only needs to hold for those $h_\ell$ that are positively $1$-homogeneous. Define
\[
  w_k := \frac{1}{r_i^{d-1} c_i} v^{(i)}_{j(i,k)} \biggl( \frac{x-x_i}{r_i} \biggr) + [u]_{Q(x_i,r_i)} \qquad
  \text{if $x \in Q_i(x_i,r_i)$, $i \in \N$,}
\]
where $[u]_{Q(x_i,r_i)} := \dashint_{Q_i(x_i,r_i)} u \dd x$. Notice that thanks to~\eqref{eq:boundary_est} we have for every $i$ that
\begin{equation} \label{eq:boundary_est2}
  \int_{\partial Q(x_i,r_i)} \absb{w_k - u} \dd \Hcal^{d-1} = \frac{1}{c_i k}
\end{equation}
We may then compute
\[
  Ew_k = \Ecal w_k \, \Lcal^d \restrict \Omega + E^s w_k,
\]
where
\[
  \Ecal w_k = \frac{1}{r_i^d c_i} \Ecal v^{(i)}_{j(i,k)} \biggl( \frac{x-x_i}{r_i} \biggr) \qquad
  \text{if $x \in Q_i(x_i,r_i)$, $i \in \N$.}
\]
Moreover, for the singular part $E^s w_k$ we can estimate, using~\eqref{eq:boundary_est2}, that
\begin{align*}
  \abs{E^s w_k}(\Omega) &\leq \sum_{i=1}^\infty \int_{\partial Q_i(x_i,r_i)} \abs{w_k-u} \dd \Hcal^{d-1} \\
  &\leq \sum_{i=1}^\infty \frac{1}{k c_i} \\
  &\leq \frac{1}{k} \bigl( \ddprb{\ONE \otimes \abs{\frarg}, \nu} + \abs{\Omega} \bigr).
\end{align*}
Here we used that $\sum_i c_i^{-1} \leq \ddprb{\ONE \otimes \abs{\frarg}, \nu} + \abs{\Omega}$ by the definition of the $c_i$'s.

In the following we will show that $Ew_k$ generates our Young measure $\nu$ that we started with. The last estimate implies that we only need to consider the Young measure generated by $\Ecal w_k$ since the singular part asymptotically vanishes. So, take $\phi_\ell \otimes h_\ell$ from the family exhibited above. We get 
\[
  \int_\Omega \phi_\ell(x) h_\ell(\Ecal w_k(x)) \dd x
  = \sum_{i=1}^\infty \int_{Q_i(x_i,r_i)} \phi_\ell(x) h_\ell \biggl( \frac{1}{r_i^d c_i} \Ecal v^{(i)}_{j(i,k)} \biggl( \frac{x-x_i}{r_i} \biggr) \biggr) \dd x.
\]

\proofstep{Step~3.} Let $x_i\in R$ be a regular point. Recall that in this case $r_i^d c_i = 1$. In the following computations $h_\ell$ can be either compactly supported (and in this case $h_\ell^\infty=0$) or positively $1$-homogeneous (and in this case $h_\ell^\infty=h_\ell$). We have for every fixed $\ell \leq k$ that
\begin{align}
  &\int_{Q_i(x_i,r_i)} \phi_\ell(x) h_\ell \biggl( \frac{1}{r_i^d c_i} \Ecal v^{(i)}_{j(i,k)} \biggl( \frac{x-x_i}{r_i} \biggr) \biggr) \dd x  \notag\\
  &\qquad = \int_{Q_i(x_i,r_i)} \phi_\ell(x_i) h_\ell \biggl( \frac{1}{r_i^d c_i} \Ecal v^{(i)}_{j(i,k)} \biggl( \frac{x-x_i}{r_i} \biggr) \biggr) \dd x  \notag\\
  &\qquad = r_i^d \phi_\ell(x_i) \int_{Q_i} h_\ell \bigl( \Ecal v^{(i)}_{j(i,k)}(y) \bigr) \dd y + E_i  \notag\\
  &\qquad = r_i^d \phi_\ell(x_i) \ddprb{\ONE_{Q_i} \otimes h_\ell, \sigma_i} + E_i  \notag\\
  &\qquad = r_i^d \int_{Q_i} \phi_\ell(x_i) \biggl( \dprb{h_\ell, \nu_{x_i}} + \dprb{h_\ell^\infty, \nu^\infty_{x_i}} \frac{\di \lambda_\nu}{\di \Lcal^d}(x_i) \biggr) \dd y + E_i  \notag\\
  &\qquad = \int_{Q_i(x_i,r_i)} \phi_\ell(x) \biggl( \dprb{h_\ell, \nu_x} + \dprb{h_\ell^\infty, \nu^\infty_x} \frac{\di \lambda_\nu}{\di \Lcal^d}(x) \biggr) \dd x + E_i.  \label{75}
\end{align}
Here $E_i$ is an error term that may change from line to line and that can be estimated as
\begin{equation}\label{eq:error estimate}
  |E_i|\le 2C_\ell  \Big(\omega_{\ell}\Big(\frac{1}{k}\Big)+\frac{1}{k}\Big) \int_{Q_i(x_i,r_i)} 1 + \absBB{\Ecal v^{(i)}_{j(i,k)} \biggl( \frac{x-x_i}{r_i} \biggr)} \dd x, \\
\end{equation}
where $ C_\ell=\|\varphi_\ell\|_\infty +\|R h_\ell\|_\infty$,  $\omega_\ell$ is a modulus of continuity for $\varphi_\ell$, and we have exploited~\eqref{eq:glue_Lebesgue1},~\eqref{eq:hl} and  that $r_i\le 1/k$. 

\proofstep{Step~4.} Let $x_i\in S$ be a singular point and let $h_\ell$ be positively $1$-homogeneous. Using~\eqref{eq:glue_Lebesgue2},~\eqref{eq:lambda_close}, we compute for every fixed $\ell \leq k$ that
\begin{align}
  &\int_{Q_i(x_i,r_i)} \phi_\ell(x) h_\ell \biggl( \frac{1}{r_i^d c_i} \Ecal v^{(i)}_{j(i,k)} \biggl( \frac{x-x_i}{r_i} \biggr) \biggr) \dd x   \notag\\
  &\qquad = \int_{Q_i(x_i,r_i)} \phi_\ell(x_i) h_\ell \biggl( \frac{1}{r_i^d c_i} \Ecal v^{(i)}_{j(i,k)} \biggl( \frac{x-x_i}{r_i} \biggr) \biggr) \dd x   \notag\\
  &\qquad = \frac{1}{c_i} \phi_\ell(x_i) \int_{Q_i} h_\ell \bigl( \Ecal v^{(i)}_{j(i,k)}(y) \bigr) \dd y + E_i  \notag\\
  &\qquad = \frac{1}{c_i} \phi_\ell(x_i) \ddprb{\ONE_{Q_i} \otimes h_\ell, \sigma_i} + E_i  \notag\\
  &\qquad = \frac{1}{c_i} \phi_\ell(x_i) \dprb{h_\ell, \nu_{x_i}^\infty} \dd \lambda_{\sigma_i}(Q_i) + E_i  \notag\\
  &\qquad = \frac{1}{c_i} \int_{Q_i} \phi_\ell(x_i) \dprb{h_\ell, \nu_{x_i}^\infty} \dd (c_i T^{x_i,r_i})_\# \lambda_\nu^s)(y) + E_i  \notag\\
  &\qquad = \int_{Q_i(x_i,r_i)} \phi_\ell(x) \dprb{h_\ell, \nu_x^\infty} \dd \lambda_\nu^s(x) + E_i,  \label{77}
\end{align}
where  the error term $E_i$ can be estimated as 
\begin{equation}\label{zencircus}
  E_i =2C_\ell\Big(\omega_{\ell}\Big(\frac{1}{k}\Big)+\frac{1}{k}\Big) \bigg(\ddprb{\ONE_{Q_i(x_i,r_i)} \otimes \abs{\frarg}, \nu}+\int_{Q_i(x_i,r_i)} \absBB{\Ecal v^{(i)}_{j(i,k)} \biggl( \frac{x-x_i}{r_i} \biggr)}\bigg).
\end{equation}
and $C_\ell$ and $\omega_\ell$ are as in~\eqref{eq:error estimate}. Here, we used that $c_i^{-1} = \ddpr{\ONE_{Q_i(x_i,r_i)} \otimes \abs{\frarg}, \nu}$ (cf.~\eqref{eq:cn}) and~\eqref{eq:hl}.

\proofstep{Step~5.}We will now show that 
\begin{equation}\label{fine}
  \int_\Omega \phi_\ell(x) h_\ell(\Ecal w_k(x)) \dd x \to \ddprb{\phi_\ell \otimes h_\ell, \nu} \qquad\text{as $k\to\infty$.}
\end{equation}
We start from 
\begin{align}
  &\int_\Omega \phi_\ell(x) h_\ell(\Ecal w_k(x)) \dd x   \notag\\
  &\qquad = \sum_{x_i \in R} \int_{Q_i(x_i,r_i)} \phi_\ell(x) h_\ell \biggl( \frac{1}{r_i^d c_i} \Ecal v^{(i)}_{j(i,k)} \biggl( \frac{x-x_i}{r_i} \biggr) \biggr) \dd x   \notag\\
&\qquad\qquad+\sum_{x_i \in S} \int_{Q_i(x_i,r_i)}\phi_\ell(x) h_\ell \biggl( \frac{1}{r_i^d c_i} \Ecal v^{(i)}_{j(i,k)} \biggl( \frac{x-x_i}{r_i} \biggr) \biggr) \dd x  \notag\\
&\qquad =: I+II  \label{76}
\end{align}
and we distinguish the cases where $h_\ell$ has compact support or is positively $1$-homogeneous.

In the first case we use~\eqref{75} and the error estimate~\eqref{eq:error estimate} to get 
\begin{equation}\label{bat1}
  I =\sum_{x_i \in R} \int_{Q_i(x_i,r_i)} \phi_\ell(x) \biggl( \dprb{h_\ell, \nu_x} + \dprb{h_\ell^\infty, \nu^\infty_x} \frac{\di \lambda_\nu}{\di \Lcal^d}(x) \biggr) \dd x+E,
\end{equation}
where $E$ can be estimated by 
\[
  |E|\le e_{k}^\ell\bigl[\abs{\Omega} + \ddprb{\ONE \otimes \abs{\frarg}, \nu} + \norm{\Ecal w_k}_{\Lrm^1} \bigr]
\]
and $e_{k}^\ell$ denotes a quantity that goes to  $0$ as $k\to \infty$ and $\ell$ fixed. For the second term we have 
\begin{equation}\label{bat2}
  | II |\le \Lcal^d\left(\bigcup_{x_i \in S} Q_i(x_i,r_i)\right) \cdot \norm{\phi_\ell \otimes h_\ell}_\infty = \hat{e}_{k}^\ell
\end{equation} 
since the union of all $Q_i(x_i,r_i)$ with $x_i \in S$ has asymptotically vanishing Lebesgue measure as $k\to \infty$. Here, again, $\hat{e}_{k}^\ell$ denotes a quantity that goes to  $0$ as $k\to \infty$ and $\ell$ fixed. Thus, combining \eqref{bat1} and \eqref{bat2}  we have shown~\eqref{fine} for $h_\ell$ compactly supported.

Let now $h_\ell$ be positively $1$-homogeneous. By using~\eqref{75} and~\eqref{eq:error estimate} the first term in~\eqref{76}  can be treated as in~\eqref{bat1} to get
\begin{equation}\label{bat3}
I=\sum_{x_i \in R} \int_{Q_i(x_i,r_i)} \phi_\ell(x) \biggl( \dprb{h_\ell, \nu_x} + \dprb{h_\ell^\infty, \nu^\infty_x} \frac{\di \lambda_\nu}{\di \Lcal^d}(x) \biggr) \dd x+e_k^\ell,
\end{equation}
where again $e_{k}^\ell\to 0$ as $k\to \infty$ and $\ell$ fixed. For the second term we note that by~\eqref{77} and~\eqref{zencircus} we have
\begin{equation}\label{bat4}
II=\sum_{x_i \in S} \int_{Q_i(x_i,r_i)} \phi_\ell(x) \dprb{h_\ell, \nu_x^\infty} \dd \lambda_\nu^s(x)+\hat{e}_{k}^\ell,
\end{equation}
where as before $\hat{e}_{k}^\ell\to 0$ as $k\to \infty$ and $\ell$ fixed.  Recalling that $h_\ell=h_\ell^\infty$ by $1$-homogeneity we deduce by~\eqref{bat3} and~\eqref{bat4} that~\eqref{fine} holds also in this case.

\proofstep{Step~6.}
Since  $\norm{\Ecal w_k}_{\Lrm^1}$ is uniformly bounded, up to a subsequence we have $\Ecal w_k \toY \mu \in \BDY(\Omega)$. However, by~\eqref{fine} also
\[
  \ddprb{\phi_\ell \otimes h_\ell, \mu}
  = \lim_{k\to\infty} \int_\Omega \phi_\ell(x) h_\ell(\Ecal w_k(x)) \dd x
  = \ddprb{\phi_\ell \otimes h_\ell, \nu}.
\]
By our choice of $\{\phi_\ell \otimes h_\ell\}_\ell$, from Lemma~\ref{lem:dense} we get $\nu = \mu \in \BDY(\Omega)$, finishing the proof.
\end{proof}

\section{Atomic parts of BD-Young measures} \label{sc:split}

As an application of the characterization theorem, we prove the following splitting result for generating sequences, a generalization of the result from Section~6 in~\cite{Rindler14YM} (the generalization can also be obtained for BV-Young measures).

\begin{theorem} \label{thm:split}
Let $\nu \in \BDY(\Omega)$ with $\lambda_\nu(\partial \Omega) = 0$ and $v \in \BD(\Omega)$. Furthermore, assume that $\nu$ has $E^s v$ as an \emph{atomic part}, that is
\begin{equation} \label{eq:atomic}
  \lambda_\nu^s \otimes \nu^\infty_x - \abs{E^s v} \otimes \delta_{\frac{\di E^s v}{\di \abs{E^s v}}(x)} \geq 0 \quad\text{in $\Mcal(\Omega \times \Rdds)$}.
\end{equation}
Then, there exists a sequence $(w_j) \subset \BD(\Omega) \cap \Crm^\infty(\Omega;\R^d)$ with $Ew_j \toweakstar [\nu] - Ev$ and such that $Ew_j + Ev \toY \nu$.
\end{theorem}

To explain this theorem, we state the following adaptation of Proposition~6 in~\cite{KristensenRindler10YM} on shifts of Young measures (the proof is the same):

\begin{lemma}[Shifts] \label{lem:shift}
Let $(u_j)$ be a bounded sequence in $\BD(\Omega)$ with $E^s u_j = 0$ and assume that $\Ecal u_j \toY \nu \in \BDY(\Omega)$. If $v \in \BD(\Omega)$, then $Eu_j + Ev \toY \mu$, where
\begin{enumerate}[(i)]
\item $\mu_x = \nu_x \conv \delta_{\Ecal v(x)}$ for $\Lcal^d$-a.e.\ $x \in \Omega$,
that is,
\[\qquad
  \dprb{h,\mu_x} = \dprb{h(\frarg + \Ecal v(x)),\nu_x}, \qquad h \in \Crm_c(\Rdds);
\]
\item $\lambda_\mu$, $(\mu_x^\infty)_x$ are such that
\[\qquad
  \dprb{f^\infty(x,\frarg),\mu_x^\infty} \, \lambda_\mu = \dprb{f^\infty(x,\frarg),\nu_x^\infty}
    \, \lambda_\nu + f^\infty \biggl( x, \frac{\di E^s v}{\di \abs{E^s v}}(x) \biggr) \, \abs{E^s v}
\]
for all $f^\infty \in \Crm(\cl{\Omega} \times \partial \Bbb^{d \times d}_\sym)$. In particular,
\[\qquad
  \lambda_\mu = \lambda_\nu + \abs{E^s v}.
\]
\end{enumerate}
\end{lemma}

However, this lemma can only be used to \emph{add} concentrations, never to \emph{remove} them. Theorem~\ref{thm:split}, however, shows that the removal of concentrations is still possible if $Ev$ is contained as an \enquote{atomic part} in $\nu$.

\begin{proof}[Proof of Theorem~\ref{thm:split}]
From~\eqref{eq:atomic} we have that for some Borel-measurable function $b \colon \Omega \to [0,1]$ it holds that
\begin{equation} \label{eq:atomic_consequences}
  \abs{E^s v} = b \lambda_\nu  \qquad\text{and}\qquad
  \nu_x^\infty \geq b(x) \delta_{\frac{\di E^s v}{\di \abs{E^s v}}(x)}.
\end{equation}
We define $\mu \in \Ybf(\Omega;\Rdds)$ for $h \in \Crm_c(\Rdds)$ through
\begin{align*}
  \dprb{h, \mu_x} &:= \dprb{f(\frarg - \Ecal v(x)), \nu_x}  \qquad \text{for $\Lcal^d$-a.e.\ $x \in \Omega$,}\\
  \mu^\infty_x    &:= \begin{cases}
  \frac{1}{1-b(x)} \Bigl( \nu^\infty_x - b(x) \delta_{\frac{\di E^s v}{\di \abs{E^s v}}(x)} \Bigr)  &\text{for $\lambda_\nu$-a.e.\ $x \in \Omega$ if $b(x) < 1$,}\\
  \nu_x^\infty = \delta_{\frac{\di E^s v}{\di \abs{E^s v}}(x)}  &\text{for $\abs{E^s v}$-a.e.\ $x \in \Omega$ if $b(x) = 1$,}
  \end{cases}\\
  \lambda_\mu     &:= \lambda_\nu(\di x) - \abs{E^s v} = (1-b(x)) \lambda_\nu(\di x).
\end{align*}
This is indeed a Young measure in $\Ybf(\Omega;\Rdds)$ by~\eqref{eq:atomic_consequences}. If $h \in \Crm(\Rdds)$ is symmetric-quasiconvex with linear growth, then define for $\Lcal^d$-almost every $x \in \Omega$ the shifted function
\[
  \tilde{h}(A) := h(A - \Ecal v(x)).
\]
Also $\tilde{h}$ is symmetric-quasiconvex with linear growth and we may estimate using the Jensen-inequality for the bulk part,~\eqref{eq:Jensen} for $\nu$, to get
\begin{align*}
  h \biggl( \dprb{\id,\mu_x} + \dprb{\id,\mu^\infty_x} \frac{\di \lambda_\mu}{\di \Lcal^d}(x) \biggr)
  &= \tilde{h} \biggl( \dprb{\id,\nu_x} + \dprb{\id,\nu^\infty_x} \frac{\di \lambda_\nu}{\di \Lcal^d}(x) \biggr) \\
  &\leq \dprb{\tilde{h},\nu_x} + \dprb{\tilde{h}^\#,\nu^\infty_x} \frac{\di \lambda_\nu}{\di \Lcal^d}(x)  \\
  &= \dprb{h,\mu_x} + \dprb{h^\#,\mu^\infty_x} \frac{\di \lambda_\mu}{\di \Lcal^d}(x)
\end{align*}
because $\tilde{h}^\# = h^\#$ and $\mu^\infty_x = \nu^\infty_x$, $\frac{\di \lambda_\mu}{\di \Lcal^d} = \frac{\di \lambda_\nu}{\di \Lcal^d}$ $\Lcal^d$-almost everywhere. Then, our main characterization result, Theorem~\ref{thm:BDY_charact}, applies and we get that $\mu \in \BDY(\Omega)$. Hence, by Lemma~\ref{lem:good_genseq}, there exists a sequence $(w_j) \subset \BD(\Omega) \cap \Crm^\infty(\Omega;\R^d)$ with $Ew_j \toY \mu$. It can be checked easily via the preceding Lemma~\ref{lem:shift} that $Ew_j + Ev$ generates $\nu$.
\end{proof}

%


\providecommand{\bysame}{\leavevmode\hbox to3em{\hrulefill}\thinspace}
\providecommand{\MR}{\relax\ifhmode\unskip\space\fi MR }
\providecommand{\MRhref}[2]{%
  \href{http://www.ams.org/mathscinet-getitem?mr=#1}{#2}
}
\providecommand{\href}[2]{#2}

\end{document}